\documentclass{amsart}
\usepackage{amssymb}
\usepackage{amsthm}
\usepackage{amsmath}

\numberwithin{equation}{section}
\newtheorem{theorem}{Theorem}[section]
\newtheorem{corollary}[theorem]{Corollary}
\newtheorem{lemma}[theorem]{Lemma}
\newtheorem{proposition}[theorem]{Proposition}
\theoremstyle{definition}
\newtheorem{definition}[theorem]{Definition}
\newtheorem{remark}[theorem]{Remark}

\title[The intersection of two real flag manifolds]
{The intersection of two real flag manifolds in a complex flag manifold}
\author[O.~Ikawa, H.~Iriyeh, T.~Okuda, T.~Sakai and H.~Tasaki]
{Osamu Ikawa, Hiroshi Iriyeh, Takayuki Okuda, Takashi Sakai and Hiroyuki Tasaki}

\address{Faculty of Arts and Sciences,
Kyoto Institute of Technology,
Matsugasaki, Sakyoku, Kyoto 606-8585, Japan}
\email{ikawa@kit.ac.jp}

\address{Graduate School of Science and Engineering,
Ibaraki University,
2-1-1, Bunkyo, Mito 310-8512, Japan}
\email{hiroshi.irie.math@vc.ibaraki.ac.jp}

\address{Mathematics Program, Graduate School of Advanced Science and Engineering,
Hiroshima University,
Hiroshima 739-8526, Japan}
\email{okudatak@hiroshima-u.ac.jp}

\address{Department of Mathematical Sciences,
Tokyo Metropolitan University,
1-1 Minami-Osawa, Hachioji, Tokyo 192–0397, Japan}
\email{sakai-t@tmu.ac.jp}

\address{Department of Mathematical Sciences,
Tokyo Metropolitan University,
1-1 Minami-Osawa, Hachioji, Tokyo 192–0397, Japan;
Department of Mathematics,
Faculty of Pure and Applied Sciences,
University of Tsukuba,
Tsukuba, Ibaraki 305-8571, Japan}
\email{tasaki@tmu.ac.jp}



\begin{document}

\maketitle

\begin{abstract}
We give a necessary and sufficient condition for two real flag manifolds,
which are not necessarily congruent,
in a complex flag manifold to intersect transversally in terms of the symmetric triad.
Then we show that the intersection of two real flag manifolds is antipodal.
As an application, we prove that any real flag manifold in a complex flag manifold
is a globally tight Lagrangian submanifold.
\end{abstract}

\section{Introduction}
\label{sec:introduction}

In 1988, Chen and Nagano \cite{Chen-Nagano1988} introduced
the notion of antipodal set of a compact symmetric space.
A subset $\mathcal{A}$ of a compact symmetric space $M$ is called an {\it antipodal set}
if the geodesic symmetry $s_x$ of $M$ at $x$ fixes all points of $\mathcal{A}$
for every point $x \in \mathcal{A}$.
The maximal possible cardinality of antipodal sets is a geometric invariant of a compact symmetric space $M$,
which is called the {\it $2$-number} and denoted by $\#_2M$.
An antipodal set of $M$ is said to be {\it great} if its cardinality attains $\#_2 M$.

In a K\"ahler manifold, a {\it real form} is 
a connected component of the fixed point set
of an anti-holomorphic involutive isometry,
which is a totally geodesic Lagrangian submanifold.
Tasaki~\cite{Tasaki2010}, Tanaka  and Tasaki~\cite{Tanaka-Tasaki2012, Tanaka-Tasaki, TTc}
studied the structure of the intersection of two real forms $L_0$ and $L_1$ in
a Hermitian symmetric space $M$ of compact type,
and proved that, if $L_0$ and $L_1$ intersect transversally,
then the intersection $L_0 \cap L_1$ is an antipodal set of $M$.
The antipodal structure of the intersection $L_0 \cap L_1$ played a crucial role
to study the Lagrangian Floer homology 
with coefficients in $\mathbb{Z}_2$
of two real forms,
which are not necessarily congruent,
in a Hermitian symmetric space of compact type (\cite{Irie-Sakai-Tasaki2013}).
Ikawa, Tanaka and Tasaki~\cite{Ikawa-Tanaka-Tasaki} gave
a necessary and sufficient condition for two real forms
in a Hermitian symmetric space of compact type to intersect transversally
in terms of the symmetric triad.
Moreover they showed that such a discrete intersection is an orbit of a certain Weyl group.

In this paper, we generalize the above results to
the case of a {\it complex flag manifold},
which is an orbit of the adjoint representation of a compact connected semisimple Lie group $G$
equipped with a $G$-invariant K\"ahler structure.
Then it is a K\"ahler $C$-space,
that is, a simply-connected compact homogeneous K\"ahler manifold (cf.\ \cite{Borel}).
Conversely, any K\"ahler $C$-space can be realized as a complex flag manifold
(cf.\ \cite{Borel}, \cite{Wang}).
In this paper, we study the intersection of two real flag manifolds in a complex flag manifold.
Using $k$-symmetric structures, 
S\'anchez \cite{Sanchez1993, Sanchez1997} studied antipodal sets of a
complex flag manifold (cf.\ \cite{Berndt-Console-Fino2001}).
In Section~\ref{sec:complex flag manifolds},
we give an alternative definition of an antipodal set of a complex flag manifold using a torus action,
and show that a maximal antipodal set is characterized as an orbit of the Weyl group of $G$
(Theorem~\ref{thm:complexflag}).

Let $(G,K)$ be a symmetric pair of compact type
with the involution $\tilde{\theta}$ of $G$,
and let $\theta$ be the differential of $\tilde{\theta}$,
which is an involution of the Lie algebra $\mathfrak{g}$ of $G$.
In this paper, we abbreviate involutive automorphism to involution.
Suppose that $-\theta$ fixes the base point $x_0 \in \mathfrak{g}$ of the adjoint orbit $M=\mathrm{Ad}(G)x_0$.
Then $-\theta$ induces an anti-holomorphic involutive isometry on the complex flag manifold $M$.
It is known that the fixed point set $L$ of $-\theta$ in $M$ is connected,
hence it is a real form of $M$.
And $K$ acts transitively on $L$
by the linear isotropy representation ($s$-representation),
that is, the restriction of the adjoint representation of $G$ to $K$.
The orbit $L = \mathrm{Ad}_G(K) x_0$
is called a {\it real flag manifold} or an {\it $R$-space}.

In Section~\ref{sec:real flag manifolds}, we summarize some facts about real flag manifolds, which are needed later.
In \cite{IST2014, IST-ICM2014},
we studied the antipodal structure of the intersection $L_0 \cap L_1$ 
of two real flag manifolds $L_0$ and $L_1$
in the complex flag manifold consisting of sequences of complex subspaces in a complex vector space
when $L_0$ and $L_1$ are congruent.

In Section~\ref{sec:intersection of real flag manifolds},
we study the intersection of two real flag manifolds,
which are not necessarily congruent, in a complex flag manifold.
Let $M=\mathrm{Ad}(G)x_0$ be a complex flag manifold
of a compact connected semisimple Lie group $G$.
Let $L_0$ and $L_1$ be two real flag manifolds of $M$
defined by anti-holomorphic involutive isometries $\tau_0=-\theta_0|_M$
and $\tau_1=-\theta_1|_M$ respectively,
where $\theta_0$ and $\theta_1$ are involutions of $\mathfrak{g}$
which satisfy $-\theta_0(x_0)=x_0$ and $-\theta_1(x_0)=x_0$.
Suppose that $\theta_0 \theta_1 = \theta_1 \theta_0$.
We set subspaces $\mathfrak{p}_i = \{ X \in \mathfrak{g} \mid -\theta_i(X) = X \} \ (i=0,1)$ in $\mathfrak{g}$,
and take a maximal abelian subspace $\mathfrak{a}$ in $\mathfrak{p}_0 \cap \mathfrak{p}_1$ containing $x_0$.
Then we have a triad $(\tilde\Sigma, \Sigma, W)$ in $\mathfrak{a}$,
where $\tilde\Sigma$ and $\Sigma$ are root systems in $\mathfrak{a}$.
A point $H \in \mathfrak{a}$ is said to be {\it regular}
if $H$ is in a {\it cell} of $\mathfrak{a}$ determined by $(\tilde\Sigma, \Sigma, W)$
(see Subsections~\ref{sec:congruent intersection} and \ref{sec:non-congruent intersection} for details).
We need the assumption $\theta_0 \theta_1 = \theta_1 \theta_0$
in order to define the triad.
Note that, in the cases of Hermitian symmetric spaces,
we may assume that $\theta_0\theta_1=\theta_1\theta_0$ without loss of generality.
Under this setting, we obtain the following theorem.

\begin{theorem}[Theorems \ref{thm:congruent} and \ref{thm:main}] \label{thm:1-1}
For $a = \exp H \ (H \in \mathfrak{a})$,
the intersection $L_0 \cap \mathrm{Ad}(a)L_1$ of $L_0$ and $\mathrm{Ad}(a)L_1$
is discrete if and only if $H \in \mathfrak{a}$ is a regular point of the triad $(\tilde{\Sigma},\Sigma,W)$.
In this case we have
$$
L_0 \cap \mathrm{Ad}(a)L_1
= M \cap \mathfrak{a}
= W(\tilde{\Sigma})x_0,
$$
where $W(\tilde{\Sigma})$ is the Weyl group of the root system $\tilde{\Sigma}$ in $\mathfrak{a}$.
In particular, a discrete intersection is an antipodal set of a complex flag manifold.
\end{theorem}

Theorem~\ref{thm:1-1} (Theorem~\ref{thm:congruent}) has an interesting application as follows.
A Lagrangian submanifold $L$ in  $M=\mathrm{Ad}(G)x_0$ is said to be {\it globally tight} (resp.\ {\it locally tight})
if for all $g \in G$ (resp.\ $g$ near the identity) such that $L$ intersects $\mathrm{Ad}(g)L$ transversally, we have
$$
\#(L \cap \mathrm{Ad}(g)L)=SB(L;\mathbb Z_2),
$$
where $SB(L;\mathbb Z_2)$ denotes the sum of the $\mathbb Z_2$-Betti numbers of $L$.
This notion was introduced by Y.~G.~Oh \cite{Oh91} in the case where $M$ is a Hermitian symmetric space of compact type,
and then extended to the case of a complex flag manifold (see \cite{Elizabeth-Luiz-Fabricio2021}).
Note that a real flag manifold
$L$ intersects $\mathrm{Ad}(g)L$ transversally
if and only if
$L \cap \mathrm{Ad}(g)L$
is discrete,
since $\dim L = \frac{1}{2}\dim M$ and $L$ is totally geodesic in $M$.
In \cite{Elizabeth-Luiz-Fabricio2021}, 
Gasparim, San Martin and Valencia
gave many examples of locally tight Lagrangian submanifolds
in complex flag manifolds.
The tightness of a Lagrangian submanifold describes the equality condition of the so-called Arnold--Givental inequality
in symplectic geometry.

As an immediate consequence of Theorem~\ref{thm:1-1}, we have the following result,
which is a generalization of Corollary~1.6 in \cite{Tanaka-Tasaki2012}.

\begin{corollary}\label{cor:globally tight}
All real flag manifolds in a complex flag manifold are globally tight Lagrangian submanifolds.
\end{corollary}
We give a proof of Corollary~\ref{cor:globally tight} in Section~\ref{sec:congruent intersection}.

\section{Complex flag manifolds and their antipodal sets}
\label{sec:complex flag manifolds}

In this section, first we recall invariant K\"ahler structures
on a complex flag manifold
(see \cite{Alekseevsky1997}, \cite{Arvanitoyeorgos} and \cite[Chapter 8]{Besse} for details).
Then we describe antipodal sets of a complex flag manifold.

Let $G$ be a compact connected semisimple Lie group
and $\mathfrak g$ its Lie algebra.
We take a non-zero vector $x_0$ in $\mathfrak{g}$
and consider its orbit
$$
M:=\mathrm{Ad}_G(G)x_0 \subset \mathfrak{g}
$$
under the adjoint representation $\mathrm{Ad}_G$ of $G$.
While there is no confusion, we write $\mathrm{Ad}$ instead of $\mathrm{Ad}_G$
for simplicity.
Let $\mathfrak{g} = \sum_{i=1}^N \mathfrak{g}_i$ be a decomposition
of the semisimple Lie algebra $\mathfrak{g}$ into a direct sum of simple ideals
$\mathfrak{g}_i \ (i=1, \ldots, N)$.
Then $M$ is decomposed into a direct product $M = M_1 \times \cdots \times M_N$,
where $M_i$ is an adjoint orbit of the connected simple Lie subgroup $G_i$ of $G$ whose Lie algebra is $\mathfrak{g}_i$.
If $\mathfrak{g}_i$-component of the base point $x_0$ of $M$ is zero,
then $M_i$ is a point.
Therefore, without loss of generality, here we may assume that
$\mathfrak{g}_i$-component of the base point $x_0$ of $M$ is non-zero for all $i$.

For each $X \in \mathfrak{g}$, let $X^*$ denote the fundamental vector field
on $M$ generated by $X$, i.e.,
$$
X^*_x := \frac{d}{dt}\Big|_{t=0} \mathrm{Ad}(\exp tX)x 
= [X, x] = -\mathrm{ad}(x)X \quad (x \in M).
$$
Since $G$ acts transitively on $M$,
the map $\mathfrak{g}\rightarrow T_x M$ defined by $X\mapsto X^*_x$ is surjective at each point $x \in M$.
Therefore $T_xM$ is expressed as
$$
T_xM = \mathrm{Im}(\mathrm{ad}(x)) \subset \mathfrak{g}
$$
for each point $x \in M$.
We give an $\mathrm{Ad}(G)$-invariant inner product $\langle \ , \ \rangle$ on $\mathfrak{g}$.
Then $M$ is regarded as a homogeneous submanifold in the Euclidean space $\mathfrak{g}$.
For each $x \in M$,
\begin{equation}\label{tangent space and normal space}
\mathfrak{g} = \mathrm{Im}(\mathrm{ad}(x)) \oplus \mathrm{Ker}(\mathrm{ad}(x))
\end{equation}
is an orthogonal direct sum decomposition with respect to $\langle \ ,\ \rangle$.
Hence $\mathrm{Ker}(\mathrm{ad}(x))$ is the normal space $T^\perp_xM$ at $x \in M$
in $\mathfrak{g}$.

Let $G_{x_0}$ be the stabilizer of $x_0$ in $G$, i.e.,
$$
G_{x_0} := \{g \in G \mid \mathrm{Ad}(g)x_0=x_0 \}.
$$
Then the adjoint orbit $M = \mathrm{Ad}(G)x_0$ is diffeomorphic
to the coset manifold $G/G_{x_0}$.
The Lie algebra $\mathfrak{g}_{x_0}$ of $G_{x_0}$ is 
the centralizer of $x_0$ in $\mathfrak{g}$, i.e.,
$$
\mathfrak{g}_{x_0}
= \{ X \in \mathfrak{g} \mid [X, x_0]=0 \}
= \mathrm{Ker}(\mathrm{ad}(x_0)).
$$
In the orthogonal direct sum decomposition (\ref{tangent space and normal space}) at $x_0$,
the subspace $\mathfrak{m} := \mathrm{Im}(\mathrm{ad}(x_0))$ is an $\mathrm{Ad}(G_{x_0})$-invariant orthogonal complement
of $\mathfrak{g}_{x_0} = \mathrm{Ker}(\mathrm{ad}(x_0))$ in $\mathfrak{g}$.
Hence $G/G_{x_0}$ is a reductive homogeneous space, and
the tangent space $T_{o}(G/G_{x_0})$ of $G/G_{x_0}$ at the origin
$o := G_{x_0}$ is also identified with $\mathfrak{m} = \mathrm{Im}(\mathrm{ad}(x_0))$.
Note that the linear transformation on $\mathfrak{m}$ defined by
$$
T_o(G/G_{x_0}) \cong
\mathfrak{m} \ni X \longmapsto X_{x_0}^* = [X,x_0] \in \mathfrak{m}=T_{x_0}(\mathrm{Ad}(G)x_0)
$$
gives the identification of $\mathfrak{m} \cong T_o(G/G_{x_0})$ with 
$\mathfrak{m}=T_{x_0}(\mathrm{Ad}(G)x_0)$.

The closure $S := \overline{\exp(\mathbb{R}x_0)}$
of the one-parameter subgroup $\exp(\mathbb{R}x_0)$ in $G$
forms a toral subgroup in $G_{x_0}$,
since $x_0$ is in $\mathfrak{g}_{x_0}$.
Then we can show that $G_{x_0}$ coincides with the centralizer $Z_G(S)$ of $S$ in $G$, i.e.,
$$
G_{x_0} = Z_G(S) := \{ g \in G \mid gh = hg \ (\forall h \in S)\},
$$
in particular $G_{x_0}$ is connected,
and $G/G_{x_0}$ is simply-connected (See \cite{Borel}).

We take a maximal torus $T$ in $G$ containing $S$.
The Lie algebra $\mathfrak{t}$ of $T$ is a maximal abelian subalgebra
of $\mathfrak{g}$.
For $\alpha \in \mathfrak{t}$, 
we define the root space $\tilde{\mathfrak{g}}_\alpha$ by
$$
\tilde{\mathfrak{g}}_\alpha := \{ X \in \mathfrak{g}^\mathbb{C} \mid 
[H,X] = \sqrt{-1}\langle\alpha, H\rangle X \ (\forall H \in \mathfrak{t})\},
$$
where $\mathfrak{g}^\mathbb{C}$ denotes the complexification of $\mathfrak{g}$.
And we set the root system of $\mathfrak{g}$ by
$\Delta := \{ \alpha \in \mathfrak{t} \setminus \{0\} \mid \tilde{\mathfrak{g}}_\alpha \neq \{0\}\}$.
Since $T$ is contained in $G_{x_0}$, $T$ is also a maximal torus in $G_{x_0}$.
Hence $\Delta_{x_0} = \{ \alpha \in \Delta \mid \langle \alpha, x_0 \rangle = 0 \}$ is
the root system of $\mathfrak{g}_{x_0}$ with respect to $\mathfrak{t}$.
This implies the disjoint union $\Delta = \Delta_{x_0} \cup \Delta_M$ and the decompositions
$$
\mathfrak{g}^{\mathbb C} = \mathfrak{g}^{\mathbb C}_{x_0} \oplus \mathfrak{m}^{\mathbb C},\qquad
\mathfrak{g}_{x_0}^\mathbb{C}
= \mathfrak{t}^\mathbb{C}
\oplus \sum_{\alpha \in \Delta_{x_0}} \tilde{\mathfrak{g}}_\alpha, \qquad
\mathfrak{m}^\mathbb{C}
= \sum_{\alpha \in \Delta_M} \tilde{\mathfrak{g}}_\alpha,
$$
where $\Delta_M := \{ \alpha \in \Delta \mid \langle \alpha, x_0 \rangle \neq 0 \}$
is the set of {\it complementary roots}.
Since $x_0$ is not equal to $0$, we have $\Delta_M\not=\emptyset$.
The center $\mathfrak{z}:=\mathfrak{z}(\mathfrak{g}_{x_0})$ of $\mathfrak{g}_{x_0}$
and the Lie algebra $\mathfrak{s}$ of $S$
satisfy ${\mathbb R}x_0 \subset \mathfrak{s} \subset \mathfrak{z} \subset \mathfrak{t}$.
The orthogonal projection from $\mathfrak{t}$
to $\mathfrak{z}$ is denoted by $H \mapsto \tilde{H}$.
The projected vector $\tilde\alpha$ of a root $\alpha \in \Delta_M$ into $\mathfrak{z}$
is called a {\it $T$-root},
and $\Delta_T$ denotes the set of all $T$-roots.
We note that $\Delta_T$ is invariant under the multiplication by $-1$,
but not a root system in $\mathfrak{z}$ in general.
Then we have direct sum decompositions
$$
\mathfrak{m}^\mathbb{C}
= \sum_{\lambda \in \Delta_T} \tilde{\mathfrak{m}}_\lambda, \quad \text{where} \quad
\tilde{\mathfrak{m}}_\lambda
:= \sum_{\alpha \in \Delta_M \atop \tilde\alpha = \lambda} \tilde{\mathfrak{g}}_\alpha,
$$
and
$$
\mathfrak{m} = \sum_{\lambda \in \Delta_T^+} \mathfrak{m}_\lambda, \quad \text{where} \quad
\mathfrak{m}_\lambda
:= (\tilde{\mathfrak{m}}_\lambda \oplus \tilde{\mathfrak{m}}_{-\lambda}) \cap \mathfrak{m},
$$
where $\Delta_T^+$ is the set of positive $T$-roots with respect to an ordering as explained below.
We can easily check that
each $\mathfrak{m}_\lambda$ is an $\mathrm{Ad}(G_{x_0})$-invariant subspace in $\mathfrak{m}$.

A connected component $C$ of the open dense subset $\mathfrak{z}_{\mathrm{reg}}$ of $\mathfrak{z}$ defined by
$$
\mathfrak{z}_\mathrm{reg} := \{ H \in \mathfrak{z} \mid \langle\lambda,H\rangle \neq 0 \ (\forall \lambda \in \Delta_T)\}
$$
is called a {\it $T$-chamber}.
For a regular element $H \in \mathfrak{z}_{\mathrm{reg}}$,
we define $\Delta_T^+ = \{ \lambda \in \Delta_T \mid \langle H,\lambda \rangle > 0 \}$
and $\Delta_T^- = \{ -\lambda \mid \lambda \in \Delta_T^+ \}$.
Then $\Delta_T$ is decomposed as
$\Delta_T = \Delta_T^+ \cup \Delta_T^-$.
Note that the definition of $\Delta_T^+$ does not depend on the choice of
regular elements in the same $T$-chamber.
We also have $\Delta_M = \Delta_M^+ \cup \Delta_M^-$,
where $\alpha \in \Delta_M$ is positive if and only if $\tilde\alpha \in \Delta_T$ is positive.
We regard $\mathfrak{m}^\mathbb{C}$ as an even dimensional real vector space,
then we define a complex structure $J$ on
$\mathfrak{m}^\mathbb{C}$ by
$$
JX = \begin{cases}
\sqrt{-1}X & (\text{if}\ X \in \tilde{\mathfrak{g}}_{\alpha}),\\
-\sqrt{-1}X & (\text{if}\ X \in \tilde{\mathfrak{g}}_{-\alpha})
\end{cases}
$$
for $\alpha \in \Delta_M^+$.
This induces an $\mathrm{Ad}(G_{x_0})$-invariant complex structure on $\mathfrak{m}$.
Since $\mathfrak{m} = T_o(G/G_{x_0})$ is isomorphic to $\mathrm{Im}(\mathrm{ad}(x_0)) = T_{x_0}M$ by
$$
\mathfrak{m} \ni X \longleftrightarrow X_{x_0}^* = [X,x_0] \in \mathrm{Im}(\mathrm{ad}(x_0)),
$$
the complex structure $J_{x_0}$ on $T_{x_0}M$ 
corresponding to $J$ on $\mathfrak{m}$ is given by
$$
J_{x_0} X_{x_0}^* = (JX)_{x_0}^* = [JX, x_0] \qquad (X \in \mathfrak{m}).
$$
Since $J$ is $\mathrm{Ad}(G_{x_0})$-invariant, we then obtain a $G$-invariant almost complex structure on $M$,
moreover we can verify that it is integrable (see \cite{Arvanitoyeorgos}, \cite{Besse}).

Hereafter we fix a $G$-invariant complex structure $J_0$,
which is called the {\it canonical complex structure},
on $M$ corresponding to the $T$-chamber
containing the base point $x_0 \in \mathfrak{z}_\mathrm{reg}$.
Then $\Delta_T^+ = \{ \lambda \in \Delta_T \mid \langle x_0, \lambda \rangle >0 \}$.
We note that $J_0$ is independent of the choice
of a maximal torus $T$ of $G$ containing $S$.
We give a $2$-form $\omega$ on $M$ by
\begin{equation} \label{eq:Kirillov-Kostant-Souriau form}
\omega(X_x^*, Y_x^*) := \langle x, [X,Y] \rangle
\end{equation}
for $x \in M$ and $X,Y \in \mathfrak{g}$.
The $2$-form $\omega$ is a $G$-invariant symplectic form on $M$,
which is called the {\it Kirillov-Kostant-Souriau form}.
Then the corresponding symmetric $2$-form $(\ , \ )$ on $M$
defined by $(\cdot, \cdot) := \omega(\cdot, J_0\cdot)$ is positive definite.
Consequently $M$ is equipped with a $G$-invariant K\"ahler structure $(J_0, \omega)$,
and we call $(M, J_0, \omega)$ a {\it complex flag manifold}.

The Ricci form $\rho$ of the complex flag manifold $(M, J_0, \omega)$ is represented as
\begin{equation} \label{eq:Ricci form}
\rho_{x_0}(X_{x_0}^*, Y_{x_0}^*) = \langle \delta, [X,Y] \rangle
\qquad (X,Y \in \mathfrak{g}),
\end{equation}
where $\delta := \sum_{\lambda \in \Delta_T^+} m_\lambda \lambda$ and $m_\lambda$ is the multiplicity
of $\lambda \in \Delta_T^+$ (see \cite[Chapter~8.C]{Besse}, \cite{Takeuchi2}).
Since $\rho$ is positive definite with respect to $J_0$,
the Ricci form $\rho$ becomes a K\"ahler form on $(M, J_0)$.
Note that each adjoint orbit admits a K\"ahler-Einstein structure
with positive Ricci curvature,
uniquely up to homothety.
Indeed, consider the adjoint orbit $M' = \mathrm{Ad}(G)\delta$ whose
base point is $\delta$.
Since $\delta$ is in $\mathfrak{z}_{\mathrm{reg}}$ and $\delta$ is in the same $T$-chamber with $x_0$,
$M'$ is diffeomorphic to $M$ and has the same $G$-invariant complex structure $J_0$.
Moreover, from the Kirillov-Kostant-Souriau form $\omega$ on $M'$
defined as (\ref{eq:Kirillov-Kostant-Souriau form})
and its Ricci form $\rho$ as (\ref{eq:Ricci form}),
we verify that 
$\omega = c \rho$ on $M'$ for some positive constant $c$ and hence
the complex flag manifold $(M', J_0, \omega)$ is K\"ahler-Einstein.

Now we shall define an antipodal set of a complex flag manifold $M$.
For a point $x$ in $M$,
let $G_x$ be the stabilizer at $x$ of $G$
and $Z(G_x)_0$ the identity component of the center $Z(G_x)$ of $G_x$.
Then $Z(G_x)_0$ is a toral subgroup of $G_x$.

\begin{definition}\label{dfn:antipodal set}
A point $y \in M$ is said to be {\it antipodal} to $x \in M$,
if $\mathrm{Ad}(g)(y) = y$ for all $g\in Z(G_x)_0$.
A subset $\mathcal{A}$ in $M$ is called an {\it antipodal set}
if $y$ is antipodal to $x$ for any $x,y \in \mathcal{A}$.
An antipodal set is said to be \textit{maximal}
if it is maximal with respect to inclusion relation.
\end{definition}

Maximal antipodal sets of $M$ can be described as follows.

\begin{theorem}\label{thm:complexflag}
For $x,y \in M$, $y$ is antipodal to $x$ if and only if $[x, y] = 0$.
For any maximal antipodal set $\mathcal{A}$ of $M$,
there exists a maximal abelian subalgebra $\mathfrak t'$ of $\mathfrak g$
such that $\mathcal{A} = M \cap \mathfrak t'$.
Hence $\mathcal{A}$ is an orbit $W(\Delta')x_1$
of the Weyl group $W(\Delta')$ of the root system $\Delta'$ of $\mathfrak{g}$
with respect to $\mathfrak t'$, where $x_1 \in \mathcal{A}$.
In particular, all maximal antipodal sets of $M$ are congruent
to each other under the action of $G$.
\end{theorem}

\begin{proof}
If $y \in M$ is antipodal to $x \in M$,
then $\mathrm{Ad}(g)y = y$ holds for all $g \in Z(G_x)_0$.
It is equivalent to $[X,y]=0$ for all $X \in \mathfrak{z}(\mathfrak{g}_x)$.
In particular, it follows $[x,y] = 0$, since $x$ is in $\mathfrak{z}(\mathfrak{g}_x)$.
Conversely, if $[x,y]=0$,
then $y$ is in $\mathfrak{g}_x$.
Therefore $[X,y]=0$ holds for all $X \in \mathfrak{z}(\mathfrak{g}_x)$,
hence $\mathrm{Ad}(g)y = y$ for all $g \in Z(G_x)_0$.

Let $\mathcal{A}$ be a maximal antipodal set of $M$.
Since $[x,y]=0$ for all $x,y \in \mathcal{A}$,
the vector subspace spanned by $\mathcal{A}$ is
an abelian subalgebra of $\mathfrak{g}$.
Thus there exists a maximal abelian subalgebra $\mathfrak{t}'$ which contains $\mathcal{A}$.
Hence $\mathcal{A} \subset M \cap \mathfrak{t}'$.
Clearly $M \cap \mathfrak{t}'$ is an antipodal set of $M$.
Therefore we have $\mathcal{A} = M \cap \mathfrak{t}'$,
since $\mathcal{A}$ is a maximal antipodal set.
It follows from \cite[p.~285, Proposition~2.2]{Helgason}
that $M \cap \mathfrak{t}'$ is an orbit $W(\Delta')x_1$
of the Weyl group $W(\Delta')$ of $\mathfrak{g}$.
From the uniqueness of a maximal abelian subalgebra of $\mathfrak{g}$ under the action of $G$,
we can verify that all maximal antipodal sets of $M$ are congruent
to each other under $G$.
\end{proof}

\begin{corollary}
For any $x, y \in M$,
$x$ is antipodal to $y$ if and only if $y$ is antipodal to $x$.
\end{corollary}

\begin{remark}
The notion of an antipodal set of a complex flag manifold $M$
was defined in \cite{IST2014}
by using $k$-symmetric structures on $M$.
Theorem~\ref{thm:complexflag} corresponds to Theorem~1 of \cite{IST2014}.
It implies that the notion of an antipodal set given by Definition~\ref{dfn:antipodal set}
is equivalent to that given in \cite{IST2014}.
In particular, when $M$ is a Hermitian symmetric space of compact type,
it is also equivalent to the notion of an antipodal set introduced in \cite{Chen-Nagano1988}.
\end{remark}

\section{Real flag manifolds in a complex flag manifold}
\label{sec:real flag manifolds}

In this section, we define a real flag manifold and show that
it is embedded in a complex flag manifold as a real form.
In this paper, for a set $\mathcal{S}$ and a map $\phi:\mathcal{S} \rightarrow \mathcal{S}$,
we set $F(\phi,\mathcal{S})=\{ x \in \mathcal{S} \mid \phi (x)=x\}$.
In the case where $G$ is a Lie group and $\phi$ is an automorphism of $G$,
let $F(\phi,G)_0$ denote the identity component of $F(\phi,G)$.

Let $G$ be a compact connected semisimple Lie group,
and let $K$ be the identity component of 
$F(\tilde\theta,G)$ for an involution $\tilde\theta$ on $G$.
Then $(G, K)$ is a symmetric pair of compact type.
The Lie algebras of $G$ and $K$ are denoted by $\mathfrak{g}$ and $\mathfrak{k}$, respectively.
The involution $\tilde\theta$ of $G$ induces an involution $\theta$ on $\mathfrak{g}$,
and we have a canonical decomposition of $\mathfrak{g}$:
$$
\mathfrak{g} = \mathfrak{k} \oplus \mathfrak{p},
$$
where $\mathfrak{p} := \{ X \in \mathfrak{g} \mid \theta(X) =-X \}
= F(-\theta, \mathfrak{g})$.

For a non-zero vector $x_0$ in $\mathfrak{p}$,
we consider a complex flag manifold $(M:=\mathrm{Ad}(G)x_0, J_0, \omega)$ in $\mathfrak{g}$
as explained in Section~\ref{sec:complex flag manifolds}.
Note that here we assume that the base point $x_0$ of $M$ has a non-zero component
in each simple ideal of $\mathfrak{g}$.
In addition, we suppose that 
the inner product $\langle \ , \ \rangle$ on $\mathfrak{g}$ is invariant under $\theta$.

\begin{proposition}\label{prop:anti-holomorphic involution}
The restriction $\tau:=-\theta|_M$ of $-\theta$ to $M$
provides an anti-holomorphic involutive isometry on the complex flag manifold $(M,J_0,\omega)$.
\end{proposition}

\begin{proof}
Since $x_0$ is in $\mathfrak{p}$, we have
$$
\tau (M)=-\theta (\mathrm{Ad}(G)x_0)=-\mathrm{Ad}(\tilde{\theta}(G))\theta (x_0)=\mathrm{Ad}(G)x_0=M.
$$
Therefore $\tau$ defines an involutive diffeomorphism of $M$.

Any tangent vector of $M$ at $x \in M$ can be expressed as $X_x^* \in T_xM$ for some $X \in \mathfrak{g}$.
Also any tangent vector of $M$ at $\tau(x) \in M$ can be expressed as
$(d\tau)_x X_x^* = (\theta(X))_{\tau(x)}^* \in T_{\tau(x)}M$.
For $X,Y \in \mathfrak{g}$, we have
\begin{align*}
\omega \big( (d\tau)_x X_x^*, (d\tau)_x Y_x^* \big)
&= \omega \big( (\theta(X))_{\tau(x)}^*, (\theta(Y))_{\tau(x)}^* \big) \\
&= \langle \tau(x), [\theta(X), \theta(Y)] \rangle \\
&= \langle -\theta(x), \theta[X,Y] \rangle \\
&= -\langle x, [X,Y] \rangle \\
&= -\omega(X_x^*, Y_x^*).
\end{align*}
Thus $\tau : M \to M$ is anti-symplectic.

Next we will show that $\tau : M \to M$ is anti-holomorphic, i.e.,
$$
(J_0)_{\tau(x)} \circ (d\tau)_x = -(d\tau)_x \circ (J_0)_{x} \qquad (x \in M).
$$
For $x \in M$, we take $g_x \in G$ such that $x = \mathrm{Ad}(g_x)x_0$.
Any tangent vector in $T_x M$ can be given as
$$
\mathrm{Ad}(g_x)X_{x_0}^*
= \mathrm{Ad}(g_x)[X,x_0]
= [\mathrm{Ad}(g_x)X,x]
= \big( \mathrm{Ad}(g_x)X \big)_x^*
\qquad (X \in \mathfrak{m}).
$$
Then we have
\begin{align*}
(J_0)_{\tau(x)} (d\tau)_x \big(\mathrm{Ad}(g_x)X\big)_x^*
&= (J_0)_{\tau(x)} \big( \theta (\mathrm{Ad}(g_x)X) \big)_{\tau(x)}^* \\
&= (J_0)_{\tau(x)} \Big( \mathrm{Ad} \big(\tilde\theta(g_x)\big) \theta(X) \Big)_{\tau(x)}^* \\
&= (J_0)_{\tau(x)} \mathrm{Ad}\big(\tilde\theta(g_x)\big) \big( \theta(X) \big)_{x_0}^* \\
&= \mathrm{Ad}\big(\tilde\theta(g_x)\big) (J_0)_{x_0} \big( \theta(X) \big)^*_{x_0} \\
&= \mathrm{Ad}\big(\tilde\theta(g_x)\big) \big( J_0 \theta(X) \big)^*_{x_0},
\end{align*}
and
\begin{align*}
(d\tau)_x (J_0)_x \big(\mathrm{Ad}(g_x)X\big)_x^*
&= (d\tau)_x (J_0)_x \mathrm{Ad}(g_x) X_{x_0}^* \\
&= (d\tau)_x \mathrm{Ad}(g_x) (J_0)_{x_0}  X_{x_0}^* \\
&= (d\tau)_x \mathrm{Ad}(g_x) (J_0X)_{x_0}^* \\
&= (d\tau)_x \big( \mathrm{Ad}(g_x) J_0X \big)_x^* \\
&= \big( \theta (\mathrm{Ad}(g_x) J_0X) \big)_{\tau(x)}^* \\
&= \Big( \mathrm{Ad}\big(\tilde\theta(g_x)\big) \theta J_0X \Big)_{\tau(x)}^* \\
&= \mathrm{Ad} \big(\tilde\theta(g_x)\big) \big( \theta J_0(X) \big)_{x_0}^*.
\end{align*}
Therefore it suffices to show that $J_0 \circ \theta = -\theta \circ J_0$ holds on $\mathfrak{m}$.
Clearly $T':=\tilde{\theta}(T)$ is also a maximal torus of $G$ containing $S$.
The Lie algebra $\mathfrak{t}'$ of $T'$ coincides with $\theta (\mathfrak{t})$. 
Let $\Delta '$ denote the root system of $\mathfrak{g}$ with respect to $\mathfrak{t}'$.
Then a root $\alpha \in \Delta$ corresponds to a root $\alpha' := \theta(\alpha) \in \Delta'$.
Moreover, $\alpha$ is in $\Delta_M^+$ if and only if $\alpha' := \theta(\alpha)$ is in $(\Delta'_M)^-$,
since
$$
\langle \alpha', x_0 \rangle
= \langle \theta(\alpha), x_0 \rangle
= \langle \alpha, \theta(x_0) \rangle
= -\langle \alpha, x_0 \rangle.
$$
For $\alpha \in \Delta_M$,
$\theta(\tilde{\mathfrak{g}}_\alpha)$ is the root space in $\mathfrak{g}^{\mathbb{C}}$
of a root $\alpha' = \theta(\alpha)$ with respect to $\mathfrak{t}' = \theta(\mathfrak{t})$.
Thus, for $\alpha \in \Delta_M^+$, we have
\begin{align*}
&J_0 \theta(X) = -\sqrt{-1} \theta(X) = -\theta(\sqrt{-1}X) = -\theta J_0 X
&&(X \in \tilde{\mathfrak{g}}_\alpha), \\
&J_0 \theta(Y) = \sqrt{-1} \theta(Y) = \theta(\sqrt{-1}Y) = -\theta J_0 Y
&&(Y \in \tilde{\mathfrak{g}}_{-\alpha}).
\end{align*}
Hence $J_0 \circ \theta = -\theta \circ J_0$ holds on
$$
\mathfrak{m} = \sum_{\alpha \in \Delta_M^+}
(\tilde{\mathfrak{g}}_\alpha \oplus \tilde{\mathfrak{g}}_{-\alpha}) \cap \mathfrak{g}.
$$
Therefore we showed that $\tau$ is anti-holomorphic.  

Consequently, since $\tau$ is anti-symplectic and anti-holomorphic,
$\tau$ is isometric for the K\"ahler metric $(\cdot, \cdot) = \omega(\cdot, J_0\cdot)$.
\end{proof}

From Proposition~\ref{prop:anti-holomorphic involution}, the fixed point set
$$
L := M \cap \mathfrak{p} = F(\tau, M)
$$
of $\tau$ in $M$ is a real form.
In fact, it is known by \cite[Proposition 2.4]{Ohnita}
that $L$ is connected and $K$ acts on $L$ transitively by the adjoint representation of $G$,
hence
$$
L = \mathrm{Ad}_G(K) x_0 \subset \mathrm{Ad}_G(G) x_0 = M.
$$
We call $\mathrm{Ad}_G(K) x_0$ a {\it real flag manifold}.

\section{The intersection of two real flag manifolds}
\label{sec:intersection of real flag manifolds}

In this section we prove that the intersection of two real flag manifolds
is antipodal in a complex flag manifold if the intersection is discrete.
Moreover we show a necessary and sufficient condition
that the intersection of two real flag manifolds is discrete,
and describe the intersection when it is discrete.
We discuss it when two real flag manifolds are
congruent and when they are non-congruent separately.
We use the notation mentioned in the previous sections.

\subsection{Antipodal property of the intersection}

In this subsection we study the antipodal
property of the intersection of two real flag manifolds
in a complex flag manifold.

A real flag manifold $L$ of a complex flag manifold $M$ is determined by an involution of
$\mathfrak g$.
Let 
$$
\mathfrak{g} = \mathfrak{k} \oplus \mathfrak{p}
$$
be the canonical decomposition of $\mathfrak g$ corresponding to the involution.
Then we have $L=M\cap\mathfrak p$, and
$$
T_xM=[\mathfrak g,x] = [\mathfrak{k},x] \oplus [\mathfrak{p},x] \quad \text{(orthogonal direct sum)}
$$
for $x \in M$.
Furthermore
$$
T_xL=[\mathfrak g,x]\cap\mathfrak p = [\mathfrak{k},x]
$$
holds for $x \in L$.

In the rest of this subsection, we discuss two real flag manifolds of 
a complex flag manifold. 
Let $G$ be a compact connected semisimple Lie group and 
$(G, K_0, K_1)$ be a compact symmetric triad, which means that 
there exist two involutions $\tilde\theta_0$ and $\tilde\theta_1$ on $G$ 
such that $K_i = F(\tilde\theta_i, G)_0 \ (i=0,1)$.
Let $\mathfrak{g}$, $\mathfrak{k}_0$ and $\mathfrak{k}_1$
denote the Lie algebras of $G$, $K_0$ and $K_1$ respectively.
Let $\mathfrak{g} = \mathfrak{k}_i \oplus \mathfrak{p}_i \ (i=0,1)$
be the canonical decompositions of $\mathfrak g$ corresponding to $\theta_i$. 
Take $x_i \in \mathfrak{p}_i\setminus \{0\}$ and assume that there 
exists an element $g_0\in G$ such that $x_0=\mathrm{Ad}_G(g_0)x_1$. 
Set $M:=\mathrm{Ad}_G(G)x_0=\mathrm{Ad}_G(G)x_1$ and 
$L_i=\mathrm{Ad}_G(K_i)x_i$. 
Then $L_0$ and $L_1$ are two real flag manifolds of the  
complex flag manifold $M$.

\begin{lemma} \label{lem:commutative}
Under the above setting,
the intersection $L_0 \cap L_1$ is non-empty.
Furthermore,
the following three conditions are equivalent. 
\begin{enumerate}
\item $L_0 \cap L_1$ is discrete. 
\item $L_0\cap L_1$ is a commutative subset of $\mathfrak{g}$. 
\item $[L_0\cap L_1, \mathfrak p_0\cap\mathfrak p_1]=\{0\}$.
\end{enumerate}
Then $L_0\cap L_1$ is an antipodal set in $M$.
In particular, if $\mathfrak{p}_0\cap\mathfrak{p}_1$ is abelian, then $L_0 \cap L_1$ is discrete. 
\end{lemma}

\begin{proof}
First we show $L_0\cap L_1\not=\emptyset$. 
We define a real flag manifold $L_1'$ which is congruent to $L_1$ by 
$L_1'=\mathrm{Ad}_G(g_0)L_1$. 
Since $L_1'=\mathrm{Ad}_G(g_0K_1g_0^{-1})x_0$, the point $x_0$ is in 
$L_0\cap L_1'$. 
Define a subspace $\mathfrak{p}_1'$ by  
$\mathfrak{p}_1'=\mathrm{Ad}_G(g_0)\mathfrak{p}_1$. 
Then $x_0$ is in $\mathfrak{p}_0\cap \mathfrak{p}_1'$. 
Take a maximal abelian subspace $\mathfrak{a}$ of 
$\mathfrak{p}_0\cap \mathfrak{p}_1'$ which contains $x_0$. 
Set $A:=\exp \mathfrak{a}$ and $K_1':=g_0K_1g_0^{-1}$. 
By a theorem of Hermann (\cite[Theorem~2.2]{Hermann}), we have 
$G=K_0 A K_1'$. 
For any $g\in G$, there exist $k_0\in K_0, a\in A$ and $k_1'\in K_1'$ 
such that $g=k_0 a k_1'$. 
Since $\mathfrak{a}$ is an abelian subspace of 
${\mathfrak p}_0\cap \mathrm{Ad}_G(a){\mathfrak p}_1'$, we have 
\begin{align*}
L_0\cap \mathrm{Ad}_G(g)L_1'
&=\mathrm{Ad}_G (k_0)(L_0\cap \mathrm{Ad}_G(a)L_1')\\
&=\mathrm{Ad}_G(k_0) 
(M\cap {\mathfrak p}_0\cap \mathrm{Ad}_G(a)(M\cap {\mathfrak p}_1'))\\
&=\mathrm{Ad}_G(k_0) (M\cap {\mathfrak p}_0\cap \mathrm{Ad}_G(a){\mathfrak p}_1')\\
&\supset \mathrm{Ad}_G(k_0) (M\cap \mathfrak{a})\ni  \mathrm{Ad}_G(k_0)x_0.
\end{align*}
If we take $g$ as $g_0^{-1}$, we have $L_0\cap L_1\not=\emptyset$. 

Next we show (1) implies (3). 
Note that $L_i=M\cap\mathfrak p_i$ $(i=0,1)$.
If $L_0\cap L_1$ is discrete, then 
for any $x\in L_0\cap L_1$,
$$
\{0\}=T_xL_0\cap T_xL_1
=[\mathfrak g,x]\cap\mathfrak p_0\cap\mathfrak p_1
$$
holds since $L_0$ and $L_1$ are totally geodesic in $M$
(see \cite[p.~61]{Kobayashi-Nomizu2}). 
For any $y\in\mathfrak p_0\cap\mathfrak p_1$,
we have $[x,y]\in\mathfrak k_0\cap\mathfrak k_1$ since $x$ is in $\mathfrak p_0\cap\mathfrak p_1$.
Take an $\mathrm{Ad}(G)$-invariant inner product $\langle \ , \ \rangle$ on $\mathfrak g$.
Then for any $X\in\mathfrak k_0\cap\mathfrak k_1$,
$$
\langle X,[x,y] \rangle = \langle [X,x],y \rangle
$$
holds.
Here $X\in\mathfrak k_0\cap\mathfrak k_1$ and $x\in\mathfrak p_0\cap\mathfrak p_1$
yield that 
$[X,x]\in\mathfrak p_0\cap\mathfrak p_1$.
On the other hand, since $[X,x]\in[\mathfrak g,x]$ we obtain
$$
[X,x]\in[\mathfrak g,x]\cap\mathfrak p_0\cap\mathfrak p_1=\{0\},
$$
which yields $[X,x]=0$.
Consequently, for any $X\in\mathfrak k_0\cap\mathfrak k_1$,
we have $\langle X,[x,y] \rangle = 0$, which means that $[x,y]=0$.
Hence we get $[L_0\cap L_1, \mathfrak p_0\cap\mathfrak p_1]=\{0\}$. 

Since $L_i$ is a subset of $\mathfrak p_i$, it is clear that (3) implies (2).

Finally, we assume (2).
By Theorem~\ref{thm:complexflag},
the intersection of $L_0$ and $L_1$ is an antipodal set of $M$,
which implies (1).
\end{proof}

\begin{remark}
The antipodal property of the intersection of two real forms
were first found in the case of the complex hyperquadrics
by Tasaki \cite{Tasaki2010}.
Before this work Howard \cite{Howard} proved that the cardinality
of the intersection of two generic real forms
in the $n$-dimensional complex projective space is
equal to $n+1$.
From a viewpoint of antipodal sets we can see the antipodal
property of the intersection by Howard's proof.
Tanaka and Tasaki~\cite{Tanaka-Tasaki2012, Tanaka-Tasaki, TTc}
proved the antipodal property of the intersection of two real forms
in Hermitian symmetric spaces of compact type
by investigating the intersections of maximal tori.
After that
\cite{Ikawa-Tanaka-Tasaki} gave
another proof by the use of symmetric triads.
In \cite{IST2014} it was
considered antipodal sets in the complex flag manifolds
and proved the antipodal property of the intersection
of two real flag manifolds in some complex flag manifolds.
The above proof of Lemma~\ref{lem:commutative}
is the simplest of the proofs mentioned in this remark.
Moreover it does not require the assumption that
two involutions determining two real flag manifolds are commutative.
\end{remark}

\subsection{Congruent case}
\label{sec:congruent intersection}

For $g \in G$ we consider the intersection $L \cap \mathrm{Ad}(g)L$
of the real flag manifold $L$ and congruent one $\mathrm{Ad}(g)L$
in a complex flag manifold $M$.
Take a maximal abelian subspace $\mathfrak{a}$ of $\mathfrak{p}$
and a maximal abelian subalgebra $\mathfrak{t}$ of $\mathfrak{g}$
satisfying $x_0 \in \mathfrak{a} \subset \mathfrak{t}$.
Set $A:=\exp\mathfrak{a}$.
Then $A$ is a toral subgroup of $G$, and $G = KAK$ holds
(\cite[Chapter~V, Theorem~6.7]{Helgason}).
Hence we can express $g = k_0ak_1$ for some $k_0, k_1 \in K$ and $a \in A$.
Since $L$ is invariant under the action of $\mathrm{Ad}(K)$,
we have
$$
L \cap \mathrm{Ad}(g)L
= L \cap \mathrm{Ad}(k_0ak_1)L
= \mathrm{Ad}(k_0)(L \cap \mathrm{Ad}(a)L).
$$
Thus it is enough to investigate $L \cap \mathrm{Ad}(a)L$.

We give an $\mathrm{Ad}(G)$-invariant inner product $\langle\cdot,\cdot\rangle$ on $\mathfrak{g}$.
For $\lambda \in \mathfrak{a}$ we define a subspace
$\mathfrak{g}_{\lambda}$ of $\mathfrak{g}^{\mathbb{C}}$ by
$$ 
\mathfrak{g}_{\lambda} := \{ X \in \mathfrak{g}^{\mathbb{C}} \mid
[H, X] = \sqrt{-1} \langle \lambda, H \rangle X \; (\forall H \in \mathfrak{a}) \},
$$
and the restricted root system
$R := \{ \lambda \in \mathfrak{a} \setminus \{0\} \mid \mathfrak{g}_{\lambda} \neq \{0\} \}$.
The orthogonal projection from $\mathfrak{t}$ to $\mathfrak{a}$
is denoted by $H \mapsto \bar{H}$.
Then
$$
R = \{ \bar{\alpha} \mid \alpha \in \Delta,\, \bar{\alpha} \neq 0 \}
$$
holds.
We define lexicographic orderings of $\mathfrak{a}$ and $\mathfrak{t}$
which are compatible with the inclusion $\mathfrak{a} \subset \mathfrak{t}$.
Let $R^+$ and $\Delta^+$ denote the sets of positive roots in $R$ and $\Delta$, respectively.
For $\lambda \in R^+$ we define
$$
\mathfrak{k}_{\lambda}
= (\mathfrak{g}_{\lambda} \oplus \mathfrak{g}_{-\lambda}) \cap \mathfrak{k}, \quad
\mathfrak{p}_{\lambda}
= (\mathfrak{g}_{\lambda} \oplus \mathfrak{g}_{-\lambda}) \cap \mathfrak{p},
$$
and
$$
\mathfrak{k}(0)
= \{ X \in \mathfrak{k} \mid [H, X] = 0\, (\forall H \in \mathfrak{a}) \}.
$$
Under this setting, we have the following lemma.

\begin{lemma}[Lemma~1.1 in \cite{Takeuchi}]\label{lem:4-1}
We have orthogonal direct sum decompositions
$$
\mathfrak{k} = \mathfrak{k}(0) \oplus \sum_{\lambda \in R^+} \mathfrak{k}_{\lambda},
\qquad
\mathfrak{p} = \mathfrak{a} \oplus \sum_{\lambda \in R^+} \mathfrak{p}_{\lambda}.
$$
For $\alpha \in \Delta^+$ with $\bar{\alpha} \neq 0$
there exist $S_{\alpha} \in \mathfrak{k}$ and $T_{\alpha} \in \mathfrak{p}$
satisfying the following conditions.
\begin{enumerate}
\item
For each $\lambda \in R^+$,
$\{ S_{\alpha} \mid \alpha \in \Delta^+, \bar{\alpha} = \lambda \}$ and
$\{ T_{\alpha} \mid \alpha \in \Delta^+, \bar{\alpha} = \lambda \}$
are orthonormal bases of $\mathfrak{k}_{\lambda}$ and $\mathfrak{p}_{\lambda}$
respectively.
\item
For each $\lambda \in R^+$, $\alpha \in \Delta^+$ with $\bar{\alpha} = \lambda \in R^+$
and $H \in \mathfrak{a}$,
the following equations hold.
\begin{align*}
&[H, S_{\alpha}] = \langle \lambda, H \rangle T_{\alpha}, \quad
[H, T_{\alpha}] = -\langle \lambda, H \rangle S_{\alpha}, \quad
[S_{\alpha}, T_{\alpha}] = \lambda, \\
& \mathrm{Ad}(\exp H)S_{\alpha} = \cos\langle\lambda,H\rangle S_{\alpha}+\sin\langle\lambda,H\rangle T_{\alpha},\\
& \mathrm{Ad}(\exp H)T_{\alpha}=-\sin\langle\lambda,H\rangle S_{\alpha}+\cos\langle\lambda,H\rangle T_{\alpha}.
\end{align*}
\end{enumerate}
\end{lemma}

Lemma~\ref{lem:4-1} implies
\begin{equation}\label{eq:4-1}
\mathfrak{p} \cap \mathrm{Ad}(a)\mathfrak{p}
= \mathfrak{a} \oplus \sum_{\lambda \in R^+ \atop \langle \lambda, H \rangle \in \pi \mathbb{Z}}
\mathfrak{p}_\lambda
\end{equation}
for any $a = \exp H \ (H \in \mathfrak a)$.
A point $H \in \mathfrak{a}$ is called a {\it regular point} of $R$
if $\langle\lambda,H\rangle\notin\pi\mathbb Z$ for any $\lambda \in R$.

\begin{theorem}\label{thm:congruent}
Under the situation mentioned above, for $a = \exp H \ (H \in \mathfrak{a})$,
the intersection $L\cap\mathrm{Ad}(a)L$ of $L$ and $\mathrm{Ad}(a)L$
is discrete if and only if $H$ is a regular point of $R$.
In this case
$$
L\cap\mathrm{Ad}(a)L
= M \cap \mathfrak a
= W(R)x_0
= W(\Delta)x_0 \cap \mathfrak{a},
$$
where $W(R)$ is the Weyl group of $R$ in $\mathfrak{a}$
and $W(\Delta)$ is the Weyl group of $\Delta$ in $\mathfrak{t}$.
\end{theorem}

\begin{proof}
The orthogonal symmetric Lie algebra $(\mathfrak{g}, \mathfrak{k})$ is decomposed
to the direct sum
$$
(\mathfrak{g}, \mathfrak{k})
= (\mathfrak{g}_1, \mathfrak{k}_1) \oplus \cdots \oplus (\mathfrak{g}_l, \mathfrak{k}_l)
$$
of irreducible orthogonal symmetric Lie algebras $(\mathfrak{g}_i, \mathfrak{k}_i) \ (i=1,\ldots, l)$
of compact type.
Note that each $(\mathfrak{g}_i, \mathfrak{k}_i)$ is of type I,
i.e. $\mathfrak{g}_i$ is a simple Lie algebra, or of type II,
i.e. $\mathfrak{g}_i = \tilde{\mathfrak{g}}_i \oplus \tilde{\mathfrak{g}}_i$ is the direct sum
of a simple Lie algebra $\tilde{\mathfrak{g}}_i$ and 
the involution on $\mathfrak{g}_i$ is given by $(X,Y) \mapsto (Y,X)$,
so that $\mathfrak{k}_i = \{ (X,X) \mid X \in \tilde{\mathfrak{g}}_i\}$
and $\mathfrak{p}_i = \{ (X,-X) \mid X \in \tilde{\mathfrak{g}}_i\}$.
Then the subspace $\mathfrak{p}$ of $\mathfrak{g}$ is decomposed to
$
\mathfrak{p} = \sum_{i=1}^l \mathfrak{p}_i,
$
where $\mathfrak{p}_i = \{ X \in \mathfrak{g}_i \mid \theta(X) = -X \}$.
The maximal abelian subspace $\mathfrak{a}$ of $\mathfrak{p}$ is also decomposed to
$
\mathfrak{a} = \sum_{i=1}^l \mathfrak{a}_i,
$
where $\mathfrak{a}_i$ is a maximal abelian subspace of $\mathfrak{p}_i$.
Moreover, the restricted root system $R$ of $(\mathfrak{g}, \mathfrak{k})$ is decomposed to
$
R = \bigcup_{i=1}^l R_i,
$
where $R_i$ is the restricted root system of $(\mathfrak{g}_i, \mathfrak{k}_i)$,
which is an irreducible root system of $\mathfrak{a}_i$.
Recall that, as mentioned in Section~\ref{sec:complex flag manifolds},
we assume that the base point $x_0$ of $M$ has a non-zero component
in each simple ideal of $\mathfrak{g}$.
Hence $x_0 \in \mathfrak{a}$ has a non-zero component in each $\mathfrak{a}_i \subset \mathfrak{p}_i \ (i=1,\ldots, l)$.

We suppose that $\langle\lambda,H\rangle\notin\pi\mathbb Z$
for any $\lambda \in R$.
The equation (\ref{eq:4-1}) implies
$$
\mathfrak p \cap\mathrm{Ad}(a)\mathfrak p
= \mathfrak a.
$$
Since $L = M\cap\mathfrak p$ we obtain
$$
L\cap\mathrm{Ad}(a)L
= (M\cap\mathfrak p)\cap\mathrm{Ad}(a)(M\cap\mathfrak p)
= M\cap(\mathfrak p\cap\mathrm{Ad}(a)\mathfrak p)
= M\cap\mathfrak a.
$$
Here $M\cap\mathfrak a$ is a subset of a maximal antipodal set $M\cap\mathfrak t$ in $M$
(Theorem~\ref{thm:complexflag}),
thus $L\cap\mathrm{Ad}(a)L$ is also antipodal in $M$,
in particular $L\cap\mathrm{Ad}(a)L$ is discrete.
Since $L=M\cap\mathfrak p$, by \cite[Chapter~VII, Proposition~2.2]{Helgason}, we have
\begin{equation}
\label{eqn:intersection of real forms}
M\cap\mathfrak a = L\cap\mathfrak a 
= \mathrm{Ad}(K)x_0\cap \mathfrak a
= W(R)x_0,
\end{equation}
and we also have
\begin{equation}
\label{eqn:intersection of real forms(2)}
M\cap\mathfrak a = (M \cap \mathfrak{t})\cap\mathfrak{a} 
= (\mathrm{Ad}(G)x_0 \cap \mathfrak{t}) \cap \mathfrak{a}
= W(\Delta)x_0 \cap \mathfrak{a}.
\end{equation}

To prove the converse, we next suppose that there is $\lambda \in R_i^+ \subset R^+$
satisfying $\langle\lambda, H\rangle \in \pi\mathbb Z$.
Since $W(R_i)\lambda$ spans $\mathfrak{a}_i$,
there exists $w \in W(R_i)$
such that $\langle w \lambda,x_0 \rangle\not= 0$. 
We set $X= w^{-1} x_0\in W(R)x_0$,
then $\langle \lambda,X\rangle \not= 0$.
For $\alpha \in \Delta$ with $\bar\alpha = \lambda$ and
$t \in \mathbb{R}$
we obtain by Lemma~\ref{lem:4-1}
\begin{equation}\label{eq:4-2}
\mathrm{Ad}(\exp tS_\alpha)X
= X + \frac{\langle\lambda, X\rangle}{|\lambda|^2}
(\cos(t|\lambda|) - 1)\lambda
- \frac{\langle\lambda, X\rangle}{|\lambda|}
\sin(t|\lambda|)T_\alpha.
\end{equation}
Since $S_\alpha$ is in $\mathfrak k$, we have
$$
\mathrm{Ad}(\exp tS_\alpha)X
\in \mathrm{Ad}(K)W(R)x_0
= \mathrm{Ad}(K)x_0
= L.
$$
Since $\langle\lambda, H\rangle \in \pi\mathbb Z$,
we have $\mathrm{Ad}(a)T_\alpha = \pm T_\alpha$.
Note that $X \in W(R)x_0 \subset \mathfrak a$ implies $\mathrm{Ad}(a)X = X$,
and $\lambda \in \mathfrak a$ implies $\mathrm{Ad}(a)\lambda = \lambda$.
From (\ref{eq:4-2}) we obtain
$$
\mathrm{Ad}(\exp tS_\alpha)X =
\begin{cases}
\mathrm{Ad}(a)\mathrm{Ad}(\exp(tS_\alpha))X \in \mathrm{Ad}(a)L, &
\text{if } \mathrm{Ad}(a)T_\alpha = T_\alpha, \\
\mathrm{Ad}(a)\mathrm{Ad}(\exp(-tS_\alpha))X \in \mathrm{Ad}(a)L, &
\text{if } \mathrm{Ad}(a)T_\alpha = -T_\alpha.
\end{cases}
$$
Consequently we have
$$
\mathrm{Ad}(\exp tS_\alpha)X
 \in L\cap\mathrm{Ad}(a)L
\quad (t \in \mathbb R)
$$
and $L\cap\mathrm{Ad}(a)L$ is not discrete.
\end{proof}

Finally we prove Corollary~\ref{cor:globally tight}.
\begin{proof}[Proof of Corollary~\ref{cor:globally tight}]
Let $M=\mathrm{Ad}(G)x_0$ be a complex flag manifold and $L$ a real flag manifold of $M$.
Then $L$ is a (totally geodesic) Lagrangian submanifold of $M$.
To examine the intersection $L \cap \mathrm{Ad}(g)L$, $g \in G$,
it suffices to consider the case that $g \in \exp \mathfrak{a}$.
For any $H \in \mathfrak{a}$ such that $L$ intersects $\mathrm{Ad}(\exp H)L$ transversally,
the intersection is discrete.
Hence, by Theorem \ref{thm:congruent},
$H$ is a regular point of $R$ and we have
$$
L \cap \mathrm{Ad}(\exp H)L
= M \cap \mathfrak{a}
= W(R)x_0.
$$
Its cardinality is $SB(L;\mathbb Z_2)$ (see \cite[p.~86]{Berndt-Console-Fino2001}),
and hence $L$ is globally tight.
\end{proof}

\subsection{Symmetric triads}
\label{sec:symmetric triads}

In the previous subsection we used one symmetric pair in order to 
treat congruent real flag manifolds. 
In order to treat non-congruent real flag manifolds we must use two symmetric pairs
which determine a symmetric triad.
In this subsection we review some results on symmetric triads
introduced in \cite{Ikawa2011} and state some further results.
These results will be used in Subsections~\ref{sec:non-congruent intersection} 
and \ref{sec:key-lemma}.

Let $G$ be a compact connected simple Lie group and
$(G, K_0, K_1)$ be a compact symmetric triad, which means that
there exist two involutions $\tilde\theta_0$ and $\tilde\theta_1$ on $G$ 
such that $K_i = F(\tilde\theta_i, G)_0 \ (i=0,1)$.
We denote by $\mathfrak{g}$, $\mathfrak{k}_0$ and $\mathfrak{k}_1$
the Lie algebras of $G$, $K_0$ and $K_1$ respectively.
We assume that $\tilde\theta_0 \tilde\theta_1 = \tilde\theta_1 \tilde\theta_0$
and that $\tilde\theta_0$ cannot be transformed to $\tilde\theta_1$
by an inner automorphism of $G$.
In Subsection~\ref{sec:non-congruent intersection},
we will construst a triad
$(\tilde{\Sigma},\Sigma,W)$ from $(G, K_0, K_1)$.
Note that the triad $(\tilde{\Sigma},\Sigma,W)$
satisfies the axiom of a symmetric triad mentioned below. 

\begin{definition}[\cite{Ikawa2011}]
Let $\mathfrak{a}$ be a finite dimensional vector space over $\mathbb{R}$
with an inner product $\langle\; ,\; \rangle$.
A triad $(\tilde\Sigma, \Sigma, W)$ is called a {\it symmetric triad} of $\mathfrak{a}$,
if it satisfies the following six conditions:
\begin{enumerate}
\item
$\tilde\Sigma$ is an irreducible root system of $\mathfrak{a}$,
and $\tilde\Sigma$ spans $\mathfrak{a}$.
\item
$\Sigma$ is a root system of $\mathrm{span}(\Sigma)$.
\item
$W$ is a nonempty subset of $\mathfrak{a}$, which is invariant
under the multiplication by $-1$,
and $\tilde\Sigma = \Sigma \cup W$.
\item
$\Sigma \cap W$ is a nonempty subset.
If we put $l = \max\{|\alpha| \mid \alpha \in \Sigma \cap W\}$,
then
$\Sigma \cap W = \{\alpha \in \tilde\Sigma \mid |\alpha| \leq l\}$.
\item
For $\alpha \in W$, $\lambda \in \Sigma \setminus W$, 
the integer $\displaystyle2\frac{\langle\alpha, \lambda\rangle}{|\alpha|^2}$
is odd if and only if $s_\alpha\lambda \in W \setminus \Sigma$,
where we set
$s_\alpha\lambda = \lambda - \displaystyle
2\frac{\langle\alpha, \lambda\rangle}{|\alpha|^2}\alpha$.
\item
For $\alpha \in W$, $\lambda \in W \setminus \Sigma$, 
the integer $\displaystyle2\frac{\langle\alpha, \lambda\rangle}{|\alpha|^2}$
is odd if and only if $s_\alpha\lambda \in \Sigma \setminus W$.
\end{enumerate}
\end{definition}

If $(\tilde\Sigma, \Sigma, W)$ is a symmetric triad of $\mathfrak{a}$,
then $\Sigma$ spans $\mathfrak{a}$ (see Subsection 4.2 in \cite{Ikawa-Tanaka-Tasaki}).

We define a nonempty subset $\mathfrak{a}_r$ in $\mathfrak{a}$ by
\begin{equation} \label{eq:regular set}
\mathfrak{a}_r = \bigcap_{\lambda\in\Sigma\atop\alpha\in W}
\left\{
H \in \mathfrak a \left|
\langle\lambda, H\rangle \notin \pi\mathbb Z,\,
\langle\alpha, H\rangle \notin \frac{\pi}2 + \pi\mathbb Z
\right.
\right\}.
\end{equation}
Then $\mathfrak a_r$ is an open dense subset of $\mathfrak a$.
A point in $\mathfrak a_r$ is called a {\it regular point}, 
and a point in ${\mathfrak a} \setminus {\mathfrak a}_r$
is called a {\it singular point} of $(\tilde{\Sigma},\Sigma, W)$.
A connected component of ${\mathfrak a}_r$
is called a {\it cell}. 
The {\it affine Weyl group} $\tilde{W}(\tilde{\Sigma},\Sigma,W)$
of $(\tilde{\Sigma},\Sigma,W)$ is defined as the subgroup 
of the semidirect product $O(\mathfrak a)\ltimes \mathfrak a$ whose generator set is given by 
$$
\left\{ \left(s_\lambda,\frac{2n\pi}{|\lambda|^2}\lambda \right) \; \Big| \; \lambda\in\Sigma,n\in\mathbb{Z}\right\}
\cup 
\left\{ \left(s_\alpha,\frac{(2n+1)\pi}{|\alpha|^2}\alpha\right) \; \Big| \; \alpha\in W,n\in\mathbb{Z}\right\}.
$$ 
It is known that $\tilde{W}(\tilde{\Sigma},\Sigma,W)$ acts transitively on the set of cells 
(Proposition 2.10 in \cite{Ikawa2011}).
Take a fundamental system $\tilde\Pi$ of $\tilde\Sigma$.
Denote by $\tilde\Sigma^+$ the set of positive roots
in $\tilde\Sigma$ with respect to $\tilde\Pi$.
If we put $\Sigma^+ = \Sigma \cap \tilde\Sigma^+$
and $W^+ = W \cap \tilde\Sigma^+$,
then $\Sigma = \Sigma^+ \cup (-\Sigma^+)$
and $W = W^+ \cup (-W^+)$ since $\Sigma$ and $W$ are invariant by the multiplication by $-1$, respectively.
Denote by $\Pi=\{\alpha_1,\ldots,\alpha_r\}$ the set of simple roots of $\Sigma$. 
By Lemma 2.12 and Theorem 2.19 in \cite{Ikawa2011} there exists a unique element 
$\tilde{\alpha}\in W^+$ such that the subset $P_0$ of $\mathfrak a$ defined by 
$$
P_0 = \left\{ H \in \mathfrak{a} \ \Big| \ 0<\langle \lambda,H\rangle \ (\forall\lambda\in \Pi),
\langle\tilde{\alpha},H\rangle<\frac{\pi}{2}\right\}
$$
is a cell. 
There exist positive integers $m_1,\ldots,m_r$ such that $\tilde{\alpha}=\sum_{i=1}^r m_i\alpha_i$ 
as mentioned on the first and second lines of page 95 in \cite{Ikawa2011}.
We can define $H_i\in\mathfrak a$ by 
$$
\langle H_i,\tilde{\alpha}\rangle=\frac{\pi}{2},\quad \langle H_i,\alpha_j\rangle =0\quad (j\not=i).
$$
Then we have $\langle H_i,\alpha_i\rangle=\frac{\pi}{2m_i}$. 
We can express $P_0$ as
$$
P_0 = \left\{ \sum_{i=1}^r t_iH_i \ \bigg| \ 0<t_i \ (1\leq^\forall\hspace{-4pt} i\leq r), \sum_{i=1}^r t_i<1\right\}.
$$
The closure of $P_0$ is a simplex.
Define a lattice $\Gamma$ of $\mathfrak a$ by 
\begin{equation}
\label{eqn:lattice}
\Gamma = \left\{H\in\mathfrak a\ \Big| \ \langle\lambda,H\rangle\in \frac{\pi}{2}\mathbb{Z}
\  (\forall\lambda\in\tilde{\Sigma})\right\},
\end{equation}
which will be used in Subsection~\ref{sec:key-lemma}.
We know, by \cite[p.~82]{Ikawa2011}, that 
$$
\Gamma=\left\{H\in\mathfrak{a} \ \Big| \ \langle \lambda,H\rangle \in\frac{\pi}{2}\mathbb{Z} \ (\forall\lambda\in \Sigma\cap W)\right\}.
$$
Since $\Sigma\cap W \subset \Sigma \subset \tilde{\Sigma}$, we have 
\begin{equation}\label{eqn:lattice2}
\Gamma = \left\{ H \in \mathfrak{a} \ \Big| \ \langle \lambda,H\rangle \in\frac{\pi}{2}\mathbb{Z} \ (\forall\lambda\in\Sigma)\right\}.
\end{equation}
By (\ref{eqn:lattice2}) we have 
$$
\Gamma = \left\{ H \in \mathfrak{a} \ \Big| \ \langle \alpha_i,H\rangle \in\frac{\pi}{2}\mathbb{Z} \
(1\leq^\forall\hspace{-4pt} i\leq r)\right\}
=\sum_{i=1}^r \mathbb{Z}m_iH_i.
$$
\begin{lemma}\label{lem:ST}
For any positive integer $n$ with $n>\sum_{i=1}^r m_i$,
there exists an element $H_0$ in ${\mathfrak a}_r$ such that  $nH_0\in \Gamma$. 
\end{lemma}
\begin{proof} Define positive numbers $t_i$ by $t_i=\frac{m_i}{n}\; (i=1,2,\ldots,r)$ and set 
$$
H_0=\sum_{i=1}^r t_iH_i=\sum_{i=1}^r \frac{m_i}{n}H_i\in \mathfrak a. 
$$
Since $t_i>0$ and $\sum_{i=1}^r t_i=\frac{1}{n}\sum_{i=1}^r m_i<1$, 
we see $H_0$ is in $P_0 (\subset \mathfrak{a}_r)$.
Furthermore we have 
$$
nH_0=\sum_{i=1}^r m_iH_i\in\Gamma.
$$
\end{proof}

\begin{lemma} \label{lem:triad-span}
Let $(\tilde{\Sigma},\Sigma,W)$ be a symmetric triad. 
Assume that $\Sigma$ is an irreducible root system. 
Then the following (1) and (2) hold. 
\begin{enumerate}
\item $W(\Sigma)\lambda$ spans $\mathfrak a$ for each $\lambda\in\Sigma$. 
\item $W(\Sigma)\alpha$ spans $\mathfrak a$ for each $\alpha\in W$. 
\end{enumerate}
\end{lemma}

\begin{proof}
The claims immediately follow from the fact that 
$W(\Sigma)$ acts irreducibly on $\mathfrak{a}$ (\cite[p.~53, Lemma~B]{Humphreys})
and $\Sigma$ spans $\mathfrak{a}$.
\end{proof}

\subsection{Non-congruent case}
\label{sec:non-congruent intersection}

In this subsection, we discuss the intersection of two real flag manifolds 
$L_0$ and $L_1$ which are not congruent in a complex flag manifold. 
Since $L_0\cap L_1\not=\emptyset$ by Lemma \ref{lem:commutative}, we can set up as follows. 

Let $G$ be a compact connected semisimple Lie group,
and $\tilde\theta_0$ and $\tilde\theta_1$ be two involutions of $G$.
Let $\mathfrak g$ denote the Lie algebra of $G$.
The differential of $\tilde\theta_i$ is denoted by $\theta_i \; (i=0,1)$.
Then we have two canonical decompositions of $\mathfrak g$: 
\begin{equation}
\label{eqn:canonical decompositions}
\mathfrak{g} = \mathfrak{k}_0 \oplus \mathfrak{p}_0 
= \mathfrak{k}_1 \oplus \mathfrak{p}_1,
\end{equation}
where $\mathfrak k_i = F(\theta_i, \mathfrak{g})$
and $\mathfrak p_i = F(-\theta_i, \mathfrak{g})$.
We can assume that $\mathfrak{p}_0\cap \mathfrak{p}_1\not=\{0\}$.
For $x_0 \in (\mathfrak{p}_0 \cap \mathfrak{p}_1) \setminus \{0\}$,
set $M = \mathrm{Ad}(G)x_0 \subset \mathfrak{g}$. 
Then $M$ is a complex flag manifold. 
For each $i=0,1$ we define a connected
closed subgroup $K_i$ of $G$ by $K_i = F(\tilde\theta_i,G)_0$. 
The Lie algebra of $K_i$ coincides with $\mathfrak{k}_i$.
We set 
\begin{equation}
\label{eqn:real form}
L_i := M \cap \mathfrak{p}_i = F(\tau_i,M) = \mathrm{Ad}(K_i) x_0 \subset M,
\end{equation}
where $\tau_i := -\theta_i|_{M}$.
Then $L_0$ and $L_1$ are real flag manifolds in $M$.

For $g\in G$ we study the intersection $L_0 \cap \mathrm{Ad}(g)L_1$ of the two real flag manifolds. 
We take a maximal abelian subspace $\mathfrak{a}$ in $\mathfrak{p}_0 \cap \mathfrak{p}_1$ containing $x_0$,
and set $A:=\exp\mathfrak{a}$, which is a toral subgroup in $G$ (cf.\ \cite[p.~107]{Ikawa2011}). 
It follows from \cite[Theorem~2.2]{Hermann} that $G=K_0 A K_1$ holds. 
Thus we can express $g = k_0ak_1$ for some $k_0 \in K_0$,\ $a \in A$ and $k_1 \in K_1$.
Since 
$$
L_0 \cap \mathrm{Ad}(g)L_1
= L_0 \cap \mathrm{Ad}(k_0ak_1)L_1
= \mathrm{Ad}(k_0)(L_0 \cap \mathrm{Ad}(a)L_1),
$$
it is enough to investigate $L_0 \cap \mathrm{Ad}(a)L_1$. 
We obtain 
\begin{equation}
\label{eqn:intersect}
L_0 \cap \mathrm{Ad}(a)L_1
= (M \cap \mathfrak{p}_0) \cap \mathrm{Ad}(a)(M \cap \mathfrak{p}_1)
= M \cap (\mathfrak{p}_0 \cap \mathrm{Ad}(a)\mathfrak{p}_1).
\end{equation}
Let $\mathfrak{a}_i$ be a maximal abelian subspace in $\mathfrak{p}_i$ containing $\mathfrak{a}$.
The maximality of $\mathfrak{a}$ in $\mathfrak{p}_0 \cap \mathfrak{p}_1$ implies
that $\mathfrak{a} = \mathfrak{a}_0 \cap \mathfrak{a}_1$. 
Denote by $R_i$ the restricted root system of $(G,K_i)$ with respect to $\mathfrak{a}_i$. 
By (\ref{eqn:intersection of real forms}), we have
\begin{equation}\label{eqn:orbit of the Weyl group}
M \cap \mathfrak{a}_i = W(R_i)x_0.
\end{equation}
In the sequel we assume that $\tilde\theta_0$ and $\tilde\theta_1$
commute with each other. 
Then we have
\begin{equation}\label{eq:simultaneous decomposition}
\mathfrak{g} = (\mathfrak{k}_0 \cap \mathfrak{k}_1)
\oplus (\mathfrak{p}_0 \cap \mathfrak{p}_1)
\oplus (\mathfrak{k}_0 \cap \mathfrak{p}_1)
\oplus (\mathfrak{p}_0 \cap \mathfrak{k}_1).
\end{equation}
We take an $\mathrm{Ad}(G)$-invariant inner product $\langle\cdot,\cdot\rangle$ on $\mathfrak{g}$. 

\begin{lemma}[Lemma~2.4 (i) of \cite{Oshima-Sekiguchi84}]
\begin{enumerate}
\item ${\mathfrak a}_0$ is $\theta_1$-invariant, and 
${\mathfrak a}_1$ is $\theta_0$-invariant. 
\item ${\mathfrak a}_0+{\mathfrak a}_1$ is abelian. 
\end{enumerate}
\end{lemma}
\begin{proof}
We give a proof for the completeness. 

(1) Since $\theta_1(\mathfrak a) = \mathfrak a$, we have
$$
[\mathfrak a,\theta_1({\mathfrak a}_0)]=\theta_1[\mathfrak a,{\mathfrak a}_0] = \{0\}.
$$
For any $H\in {\mathfrak a}_0$, there exist
$H_1$ in ${\mathfrak p}_0 \cap \mathfrak{p}_1$ and $H_2$ in
${\mathfrak p}_0\cap {\mathfrak k}_1$ such that $H=H_1+H_2$ since
${\mathfrak p}_0=({\mathfrak p}_0\cap {\mathfrak p}_1)\oplus ({\mathfrak p}_0\cap {\mathfrak k}_1)$.
For any $X\in \mathfrak a$, we have 
$$
0=[X,\theta_1(H)]=-[X,H_1]+[X,H_2]. 
$$
Since $[X,H_1]$ is in ${\mathfrak k}_1$ and $[X,H_2]$ is in ${\mathfrak{p}}_1$, 
we have $[X,H_1]=0$. 
The maximality of $\mathfrak a$ implies that $H_1$ is in 
$\mathfrak a \subset {\mathfrak a}_0$.
Hence $H_2 = H - H_1$ is in $\mathfrak{a}_0$,
and so $\theta_1(H) = -H_1 + H_2$ is in $\mathfrak{a}_0$.
Thus ${\mathfrak a}_0$ is $\theta_1$-invariant. 
Similarly ${\mathfrak a}_1$ is $\theta_0$-invariant. 

(2) Since ${\mathfrak a}_0$ and ${\mathfrak a}_1$ are $\theta_1$ and 
$\theta_0$-invariant, respectively, we have
$$
{\mathfrak a}_0=\mathfrak a\oplus ({\mathfrak a}_0\cap {\mathfrak k}_1),\quad 
{\mathfrak a}_1=\mathfrak a\oplus ({\mathfrak a}_1\cap {\mathfrak k}_0).
$$
It is sufficient to prove that 
$[{\mathfrak a}_0\cap {\mathfrak k}_1,{\mathfrak a}_1\cap {\mathfrak k}_0] = \{0\}$.
By the Jacobi identity, we have 
$[\mathfrak a,[{\mathfrak a}_0\cap {\mathfrak k}_1,{\mathfrak a}_1\cap {\mathfrak k}_0]]
=\{0\}$. 
Since $[{\mathfrak a}_0\cap {\mathfrak k}_1,{\mathfrak a}_1\cap {\mathfrak k}_0]
\subset \mathfrak{p}_0 \cap \mathfrak{p}_1$,
and $\mathfrak{a}$ is a maximal abelian subspace in $\mathfrak{p}_0 \cap \mathfrak{p}_1$,
we have $[{\mathfrak a}_0\cap {\mathfrak k}_1,{\mathfrak a}_1\cap {\mathfrak k}_0] \subset \mathfrak{a}$.
Since $\langle\;,\;\rangle$ is $\mathrm{Ad}(G)$-invariant, we have 
$\langle
\mathfrak a,[{\mathfrak a}_0\cap {\mathfrak k}_1,{\mathfrak a}_1\cap {\mathfrak k}_0]
\rangle=\{0\}.$
Thus we get 
$[{\mathfrak a}_0\cap {\mathfrak k}_1,{\mathfrak a}_1\cap {\mathfrak k}_0]=\{0\}$.
\end{proof}

By the lemma above, 
we see that $[\mathfrak{a}_0,\mathfrak{a}_1] = \{0\}$,
and therefore if we take a maximal abelian subalgebra $\mathfrak{t}$ of 
$\mathfrak{g}$ containing $\mathfrak{a}_0+\mathfrak{a}_1$ then 
it satisfies the following condition:
\begin{itemize}
\item $\mathfrak{t}$ is invariant under both $\theta_0$ and $\theta_1$,
and $\mathfrak{a}_i = \mathfrak t \cap \mathfrak{p}_i$ for $i = 0,1$.
\end{itemize}
Note that $\mathfrak{a} \subset \mathfrak{t}$.
Let $\Delta$ denote the root system of $(\mathfrak{g}, \mathfrak{t})$
and $W(\Delta)$ the Weyl group of $\Delta$ acting on $\mathfrak{t}$.
For each $\alpha \in \mathfrak{a}$ define a subspace 
$\mathfrak{g}(\mathfrak{a}, \alpha)$ of $\mathfrak{g}^{\mathbb{C}}$ by
$$
\mathfrak{g}(\mathfrak{a}, \alpha)
= \{X \in \mathfrak{g}^{\mathbb{C}} \mid
[H, X] = \sqrt{-1}\langle\alpha, H\rangle X\; (\forall H \in \mathfrak{a})\}
$$
and set $\tilde\Sigma = \{\alpha \in \mathfrak{a} \setminus \{0\} \mid
\mathfrak{g}(\mathfrak{a}, \alpha) \neq \{0\}\}$.
It is known that $\tilde{\Sigma}$ is a root system which spans 
$\mathfrak a$ (\cite{Ikawa2011}, Lemma~4.12).
Let $W(\tilde{\Sigma})$ denote the Weyl group of $\tilde{\Sigma}$ acting on $\mathfrak{a}$.
For $\epsilon = \pm 1$ define a subspace
$\mathfrak{g}(\mathfrak{a}, \alpha, \epsilon)$ of
$\mathfrak{g}(\mathfrak{a}, \alpha)$ by
$$
\mathfrak{g}(\mathfrak{a}, \alpha, \epsilon)
= \{X \in \mathfrak{g}(\mathfrak{a}, \alpha) \mid \theta_0 \theta_1 X = \epsilon X\}.
$$
Since $\mathfrak{g}(\mathfrak{a}, \alpha)$ is
$\theta_0 \theta_1$-invariant,
we have
$$
\mathfrak{g}(\mathfrak{a}, \alpha)
= \mathfrak{g}(\mathfrak{a}, \alpha, 1) \oplus \mathfrak{g}(\mathfrak{a}, \alpha, -1).
$$
Set $\Sigma = \{\alpha \in \tilde\Sigma \mid \mathfrak{g}(\mathfrak{a}, \alpha, 1) \neq \{0\}\}$
and $W = \{\alpha \in \tilde\Sigma \mid \mathfrak{g}(\mathfrak{a}, \alpha, -1) \neq \{0\}\}$.
Then $\tilde{\Sigma} = \Sigma\cup W$.
A point $H\in\mathfrak{a}$ is called a {\it regular point} with respect to $(\tilde{\Sigma},\Sigma,W)$
if $\langle\lambda,H\rangle\not\in\pi\mathbb{Z}$ for any $\lambda\in\Sigma$ and 
$\langle\alpha,H\rangle\not\in \frac{\pi}{2}+\pi\mathbb{Z}$ for any $\alpha\in W$.
A point $H\in\mathfrak a$ is called a {\it singular point} if $H$ is not a 
regular point.

If $L_0$ and $L_1$ are not congruent in $M$,
then $\tilde{\theta}_0$ and $\tilde{\theta}_1$ are not conjugate by the inner automorphisms.
By Proposition 4.39 of \cite{Ikawa2011}, we have $\Sigma \cap W \neq \emptyset$. 
In the sequel, we assume that $L_0$ and $L_1$ are not congruent.

\begin{theorem}\label{thm:main}
Under the situation mentioned above,
for $a =\exp H\; (H\in \mathfrak{a})$, 
the intersection $L_0 \cap \mathrm{Ad}(a)L_1$ of $L_0$ and $\mathrm{Ad}(a)L_1$ 
is discrete if and only if $H$ is a regular point of the triad $(\tilde{\Sigma},\Sigma,W)$.
In this case
$$
L_0 \cap \mathrm{Ad}(a)L_1
= M \cap \mathfrak{a}
= W(\tilde{\Sigma})x_0
= W(R_i)x_0 \cap \mathfrak{a}
= W(\Delta)x_0 \cap \mathfrak{a}
$$
for $i=0,1$.
\end{theorem}

Define closed subgroups $G_{01}$ and $K_{01}$ of $G$ by
$G_{01} = F(\tilde\theta_0 \tilde\theta_1, G)_0$ and
$K_{01} = F(\tilde\theta_0, G_{01})$.
Then the Lie algebras of $G_{01}$ and $K_{01}$ are given by
$$
\mathfrak g_{01}
= (\mathfrak{k}_0 \cap \mathfrak{k}_1) \oplus (\mathfrak{p}_0 \cap \mathfrak{p}_1), \quad
\mathfrak k_{01} = \mathfrak{k}_0 \cap \mathfrak{k}_1,
$$
respectively.
The restricted root system of the compact symmetric pair $(G_{01}, K_{01})$
with respect to $\mathfrak{a}$ coincides with $\Sigma$.
In particular, $\Sigma$ is a root system. 
For $\lambda \in \Sigma$, we define subspaces $\mathfrak{p}_\lambda$
in $\mathfrak{p}_0 \cap \mathfrak{p}_1$ and
$\mathfrak{k}_\lambda$ in $\mathfrak{k}_0 \cap \mathfrak{k}_1$
as follows:
\begin{align*}
\mathfrak{p}_\lambda
& = \{X \in \mathfrak{p}_0 \cap \mathfrak{p}_1
\mid [H, [H, X]] = - \langle\lambda, H\rangle^2 X\; (H \in \mathfrak{a})\} \\
& = \big( \mathfrak{g}(\mathfrak{a}, \lambda, 1) \oplus \mathfrak{g}(\mathfrak{a}, -\lambda, 1) \big)
\cap (\mathfrak{p}_0 \cap \mathfrak{p}_1), \\
\mathfrak{k}_\lambda
& = \{X \in \mathfrak{k}_0 \cap \mathfrak{k}_1
\mid [H, [H, X]] = - \langle\lambda, H\rangle^2 X\; (H \in \mathfrak{a})\}\\
& = \big( \mathfrak{g}(\mathfrak{a}, \lambda, 1) \oplus \mathfrak{g}(\mathfrak{a}, -\lambda, 1) \big)
\cap (\mathfrak{k}_0 \cap \mathfrak{k}_1).
\end{align*}

Define a subalgebra $\mathfrak{k}(0)$ 
in $\mathfrak{k}_0 \cap \mathfrak{k}_1$ by
$$
\mathfrak{k}(0)
= \{X \in \mathfrak{k}_0 \cap \mathfrak{k}_1 \mid [\mathfrak{a}, X] = \{0\}\}.
$$
Then we have orthogonal direct sum decompositions:
$$
\mathfrak{k}_0 \cap \mathfrak{k}_1
= \mathfrak{k}(0) \oplus \sum_{\lambda \in \Sigma^+} \mathfrak{k}_\lambda, \quad
\mathfrak{p}_0 \cap \mathfrak{p}_1
= \mathfrak{a} \oplus \sum_{\lambda \in \Sigma^+} \mathfrak{p}_\lambda.
$$
Define subspaces of $\mathfrak{k}_0 \cap \mathfrak{p}_1$
and $\mathfrak{p}_0 \cap \mathfrak{k}_1$ by
\begin{align*}
V(\mathfrak{k}_0 \cap \mathfrak{p}_1)
& = \{X \in \mathfrak{k}_0 \cap \mathfrak{p}_1 \mid
[\mathfrak a, X] = \{0\}\}, \\
V(\mathfrak{p}_0 \cap \mathfrak{k}_1)
& = \{X \in \mathfrak{p}_0 \cap \mathfrak{k}_1 \mid
[\mathfrak a, X] = \{0\}\}, \\
V^\perp(\mathfrak{k}_0 \cap \mathfrak{p}_1)
& = \{X \in \mathfrak{k}_0 \cap \mathfrak{p}_1 \mid
X \perp V(\mathfrak{k}_0 \cap \mathfrak{p}_1)\}, \\
V^\perp(\mathfrak{p}_0 \cap \mathfrak{k}_1)
& = \{X \in \mathfrak{p}_0 \cap \mathfrak{k}_1 \mid
X \perp V(\mathfrak{p}_0 \cap \mathfrak{k}_1)\}.
\end{align*}
For $\alpha \in W$ define subspaces
$V_\alpha^\perp(\mathfrak{k}_0 \cap \mathfrak{p}_1)$
in $V^\perp(\mathfrak{k}_0 \cap \mathfrak{p}_1)$, and
$V_\alpha^\perp(\mathfrak{p}_0 \cap \mathfrak{k}_1)$
in $V^\perp(\mathfrak{p}_0 \cap \mathfrak{k}_1)$ by
\begin{align*}
V_\alpha^\perp(\mathfrak{k}_0 \cap \mathfrak{p}_1)
& = \{X \in V^\perp(\mathfrak{k}_0 \cap \mathfrak{p}_1) \mid
[H, [H, X]] = -\langle\alpha, H\rangle^2 X\; (H \in \mathfrak{a})\} \\
& = \big( \mathfrak{g}(\mathfrak{a}, \alpha, -1) \oplus \mathfrak{g}(\mathfrak{a}, -\alpha, -1) \big)
\cap (\mathfrak{k}_0 \cap \mathfrak{p}_1), \\
V_\alpha^\perp(\mathfrak{p}_0 \cap \mathfrak{k}_1)
& = \{X \in V^\perp(\mathfrak{p}_0 \cap \mathfrak{k}_1) \mid
[H, [H, X]] = -\langle\alpha, H\rangle^2 X\; (H \in \mathfrak{a})\} \\
& = \big( \mathfrak{g}(\mathfrak{a}, \alpha, -1) \oplus \mathfrak{g}(\mathfrak{a}, -\alpha, -1) \big)
\cap (\mathfrak{p}_0 \cap \mathfrak{k}_1).
\end{align*}
Then we have the orthogonal direct sum decompositions:
$$
V^\perp(\mathfrak{k}_0 \cap \mathfrak{p}_1)
= \sum_{\alpha \in W^+} V_\alpha^\perp(\mathfrak{k}_0 \cap \mathfrak{p}_1),
\quad
V^\perp(\mathfrak{p}_0 \cap \mathfrak{k}_1)
= \sum_{\alpha \in W^+} V_\alpha^\perp(\mathfrak{p}_0 \cap \mathfrak{k}_1).
$$
For $\lambda \in \Sigma$ and $\alpha \in W$, set
\begin{align*}
m(\lambda)
&:= \dim_{\mathbb C} \mathfrak g(\mathfrak a, \lambda, 1)
= \dim_{\mathbb{R}} \mathfrak{p}_{\lambda}
= \dim_{\mathbb{R}} \mathfrak{k}_{\lambda}, \\
n(\alpha)
&:= \dim_{\mathbb C} \mathfrak g(\mathfrak a, \alpha, -1)
= \dim_{\mathbb{R}} V_\alpha^\perp(\mathfrak{k}_0 \cap \mathfrak{p}_1)
= \dim_{\mathbb{R}} V_\alpha^\perp(\mathfrak{p}_0 \cap \mathfrak{k}_1).
\end{align*}

\begin{lemma}[Lemma 4.16 in \cite{Ikawa2011}]\label{lem:bracket of triad}
\begin{enumerate}
\item
For any $\alpha \in W^+$,
\begin{align*}
[\mathfrak{a}, V_\alpha^\perp(\mathfrak{k}_0 \cap \mathfrak{p}_1)]
& = V_\alpha^\perp(\mathfrak{p}_0 \cap \mathfrak{k}_1), \\
[\mathfrak{a}, V_\alpha^\perp(\mathfrak{p}_0 \cap \mathfrak{k}_1)]
& = V_\alpha^\perp(\mathfrak{k}_0 \cap \mathfrak{p}_1).
\end{align*}
\item
For each $\alpha \in W^+$,
there exist orthonormal bases
$\{X_{\alpha,i}\}_{1 \leq i \leq n(\alpha)}$ and
$\{Y_{\alpha,i}\}_{1 \leq i \leq n(\alpha)}$ of
$V_\alpha^\perp(\mathfrak{k}_0 \cap \mathfrak{p}_1)$ and
$V_\alpha^\perp(\mathfrak{p}_0 \cap \mathfrak{k}_1)$ respectively
such that, for any $H \in \mathfrak{a}$,
\begin{align*}
& [H, X_{\alpha,i}] = \langle\alpha, H\rangle Y_{\alpha,i}, \quad
[H, Y_{\alpha,i}] = -\langle\alpha, H\rangle X_{\alpha,i}, \\
& [X_{\alpha,i}, Y_{\alpha,i}] = \alpha,\\
& \mathrm{Ad}(\exp H)X_{\alpha,i}=\cos\langle\alpha,H\rangle X_{\alpha,i}
+\sin\langle\alpha,H\rangle Y_{\alpha,i},\\
& \mathrm{Ad}(\exp H)Y_{\alpha,i}=-\sin\langle\alpha,H\rangle X_{\alpha,i}
+\cos\langle\alpha,H\rangle Y_{\alpha,i}.
\end{align*}
\end{enumerate}
\end{lemma}

We know that 
\begin{equation}\label{eq:Weyl gr of tildeSigma}
W(\tilde\Sigma) \subset \{\mathrm{Ad}(g)|_{\mathfrak a} \mid
g \in G,\, \mathrm{Ad}(g)\mathfrak a = \mathfrak a\}.
\end{equation}
See Corollary 4.17 and Lemma 4.4 in \cite{Ikawa2011}
for the detail.

The triple $(\mathfrak g,\theta_0,\theta_1)$ is said to be {\it irreducible} if there exists no nontrivial ideal 
of $\mathfrak g$ which is invariant under $\theta_0$ and $\theta_1$. 
The triple $(G,K_0,K_1)$ is said to be {\it irreducible} if 
$(\mathfrak g,\theta_0,\theta_1)$ is irreducible.

Now we are in a position to prove Theorem~\ref{thm:main}. 

\begin{proof}[Proof of Theorem~\ref{thm:main}]
The triple $(\mathfrak g,\theta_0,\theta_1)$ can be decomposed 
to the direct sum of irreducible factors. 
Recall that, as mentioned in Section~\ref{sec:complex flag manifolds},
$x_0\in \mathfrak{a}$ has a non-zero component in each simple ideal of $\mathfrak g$.
Thus it suffices to prove the theorem for each irreducible component (see Subsection~\ref{sec:general case} for details).

In what follows we assume that $(\mathfrak g,\theta_0,\theta_1)$ 
is irreducible.
By Lemma~\ref{lem:bracket of triad} we have
$$
\mathfrak{p}_0 \cap \mathrm{Ad}(a)\mathfrak{p}_1
= \mathfrak{a} \oplus
\sum_{\lambda\in\Sigma^+\atop\langle\lambda, H\rangle\in\pi\mathbb Z}
\mathfrak{p}_\lambda \oplus
\sum_{\alpha\in W^+\atop\langle\alpha, H\rangle\in\frac{\pi}2+\pi\mathbb Z}
V_\alpha^\perp(\mathfrak{p}_0 \cap \mathfrak{k}_1),
$$
where $a=\exp H$.
Recall that a point $H\in\mathfrak{a}$ is a regular point with respect to $(\tilde{\Sigma},\Sigma,W)$
if $\langle\lambda,H\rangle\not\in\pi\mathbb{Z}$ for each $\lambda\in\Sigma$ and 
$\langle\alpha,H\rangle\not\in \frac{\pi}{2}+\pi\mathbb{Z}$ for each $\alpha\in W$, 
that is equivalent to the condition 
${\mathfrak p}_0\cap \mathrm{Ad}(a){\mathfrak p}_1=\mathfrak a$.
By using (\ref{eqn:intersect}), we have
$$
L_0 \cap \mathrm{Ad}(a)L_1
= M \cap \left(\mathfrak{a} \oplus
\sum_{\lambda\in\Sigma^+\atop\langle\lambda, H\rangle\in\pi\mathbb Z}
\mathfrak{p}_\lambda \oplus
\sum_{\alpha\in W^+\atop\langle\alpha, H\rangle\in\frac{\pi}2+\pi\mathbb Z}
V_\alpha^\perp(\mathfrak{p}_0 \cap \mathfrak{k}_1)\right).
$$

If $H$ is a regular point of $(\tilde{\Sigma}, \Sigma ,W)$,
then by (\ref{eqn:orbit of the Weyl group})
$$
L_0 \cap \mathrm{Ad}(a)L_1
= M \cap \mathfrak{a}
= (M\cap \mathfrak{a}_i) \cap \mathfrak{a} 
= W(R_i) x_0 \cap \mathfrak{a} \quad (i = 0,1),
$$
which implies that the intersection is discrete. 

We assume that $H$ is a singular point of $(\tilde{\Sigma}, \Sigma, W)$. 
We will show that the intersection $L_0 \cap \mathrm{Ad}(a)L_1$ is not discrete. 
By the assumption, there are two cases:
\begin{enumerate}
\item[(i)] there exists $\lambda \in \Sigma$ such that $\langle\lambda ,H\rangle\in \pi\mathbb{Z}$,
\item[(ii)] there exists $\alpha \in W$ such that $\langle\alpha ,H\rangle \in \frac{\pi}{2}+\pi\mathbb{Z}$. 
\end{enumerate}
In both cases, we may assume that $W(\Sigma)\lambda$ spans $\mathfrak{a}$ for each 
$\lambda\in \Sigma\cup W$ by Lemma~\ref{lem:span} mentioned below.
In the case of (i) we can show that the intersection $L_0 \cap \mathrm{Ad}(a)L_1$ is not discrete
in a similar manner of the proof of Theorem~\ref{thm:congruent}.
In the case of (ii) there exists $X \in W(\Sigma) x_0$ such that $\langle\alpha, X\rangle \neq 0$
since $W(\Sigma)\alpha$ spans $\mathfrak{a}$.
By Lemma~\ref{lem:bracket of triad} we have 
$$
\mathrm{Ad}(\exp tX_{\alpha ,i})X=X+\frac{\langle\alpha ,X\rangle}{|\alpha|^2}(\cos (t|\alpha|)-1)\alpha 
-\frac{\langle\alpha ,X\rangle}{|\alpha|}\sin (t|\alpha|)Y_{\alpha ,i}.
$$
Since the left-hand side is a curve in $M$,
and the right-hand side is a circle in $\mathfrak{p}_0 \cap \mathrm{Ad}(a)\mathfrak{p}_1$, 
the curve $\mathrm{Ad}(\exp tX_{\alpha ,i})X$ is a circle in $L_0 \cap \mathrm{Ad}(a)L_1$. 
Hence the intersection $L_0 \cap \mathrm{Ad}(a)L_1$ is not discrete. 

Finally, we shall prove that $W(\tilde{\Sigma}) x_0 = M \cap \mathfrak{a}$
where $M = \mathrm{Ad}(G) x_0$
(Note that this equality holds for any $x_0 \in \mathfrak{a}\setminus \{0\}$).
Since $x_0$ is in $\mathfrak{a}$, we have
$W(\tilde{\Sigma})x_0\subset \mathfrak{a}$.
The inclusion (\ref{eq:Weyl gr of tildeSigma}) implies $W(\tilde{\Sigma})x_0\subset \mathrm{Ad}(G)x_0=M$.
Hence $W(\tilde{\Sigma})x_0\subset M\cap \mathfrak{a}$.
Since $M$ is $W(\tilde{\Sigma})$-invariant by (\ref{eq:Weyl gr of tildeSigma}), 
in order to prove $W(\tilde{\Sigma}) x_0 = M \cap \mathfrak{a}$,
it suffices to show that $M \cap \mathfrak{a}$ is a single $W(\tilde{\Sigma})$-orbit.
Recall that $M\cap\mathfrak{t}=\mathrm{Ad}(G)x_0\cap \mathfrak{t}$ 
coincides with a single $W(\Delta)$-orbit $W(\Delta)x_0$ by Theorem~\ref{thm:complexflag}.
By (3.8.4) in \cite{Oshima-Sekiguchi84},
there exists a positive system $\Delta^+$ of $\Delta$ such that
$$
\tilde{\Sigma}^+ := \big\{ \bar{\alpha} \ \big| \ \alpha \in \Delta^+ \big\} \setminus \{  0 \}
$$
is a positive system of $\tilde{\Sigma}$,
where $\bar{\alpha} \in \mathfrak{a}$ is the orthogonal projection of $\alpha \in \mathfrak{t}$.
We set
\begin{align*}
\mathfrak{a}_+ &:= \{ X \in \mathfrak{a} \mid \langle \lambda, X \rangle \geq 0 \text{ for any } \lambda \in \tilde{\Sigma}^+ \}, \\
\mathfrak{t}_+ &:= \{ X \in \mathfrak{t} \mid \langle \alpha, X \rangle \geq 0 \text{ for any } \alpha \in \Delta^+ \}.
\end{align*}
Then $\mathfrak{a}_+$ and $\mathfrak{t}_+$ are closed fundamental domains of the $W(\tilde{\Sigma})$-action 
on $\mathfrak{a}$
and the $W(\Delta)$-action on $\mathfrak{t}$, respectively.
By the definition of $\tilde{\Sigma}^+$, we have $\mathfrak{a}_+ \subset \mathfrak{t}_+$.
Since $M \cap \mathfrak{t}$ is a single $W(\Delta)$-orbit,
we see that $\# (M \cap \mathfrak{t}_+) = 1$
by \cite[p.~293, Theorem 2.22]{Helgason}
and hence $\# (M \cap \mathfrak{a}_+) \leq 1$.
This implies that $M \cap \mathfrak{a}$ consists of at most one $W(\tilde{\Sigma})$-orbit.
Therefore we have $W(\tilde{\Sigma}) x_0 = M \cap \mathfrak{a} 
= (M\cap \mathfrak t)\cap \mathfrak a = W(\Delta)x_0 \cap \mathfrak a$.
\end{proof}

\subsection{Technical lemmas}
\label{sec:key-lemma}

In this subsection, we establish Lemma~\ref{lem:span} below, which was used in Subsection~\ref{sec:non-congruent intersection}.
We also prove Lemma~\ref{lem:regularity} below,
which is not used in this paper but seems to be useful for further geometric applications of symmetric triads.

Let $G$ be a compact connected semisimple Lie group, and $(G,K_0,K_1)$ be a compact symmetric triad. 
Denote by $\tilde\theta_i$ the involution of $G$ which defines $K_i$ for $i=0,1$. 
We assume that $\tilde\theta_0$ and $\tilde\theta_1$ commute with each other. 
We decompose the Lie algebra $\mathfrak g$ of $G$ as in (\ref{eqn:canonical decompositions}). 
Take and fix a maximal abelian subspace $\mathfrak a$ of ${\mathfrak p}_0\cap {\mathfrak p}_1$. 
For $x_0\in \mathfrak a\setminus\{0\}$ the orbit $M=\mathrm{Ad}(G)x_0$ is a complex flag manifold
whose real dimension is greater than or equal to $2$.
For $H\in\mathfrak a$ we set $a=\exp H$. 
We define an involution $\theta_1'$ on $\mathfrak g$
by $\theta_1' = \mathrm{Ad}(a) \theta_1 \mathrm{Ad}(a^{-1})$.

We retain the setting and notations as in Subsection~\ref{sec:non-congruent intersection}. 
The purpose of this subsection is to show the following two lemmas.

\begin{lemma}\label{lem:regularity}
There exist a regular point 
$H\in \mathfrak{a}$ and $n\in\mathbb{N}$ such that 
$(\theta_0 \theta_1')^{2n}$ is the identity mapping on $M$.
\end{lemma}

By the following lemma,
we can suppose the conditions (3-2) and (3-3) below
to prove Theorem~\ref{thm:main} without loss of generality.

\begin{lemma}\label{lem:span}
Let $(G,K_0,K_1)$ be an irreducible compact symmetric triad which satisfies $\tilde\theta_0 \tilde\theta_1 = \tilde\theta_1 \tilde\theta_0$.
Then there exists $g\in G$ which satisfies the following conditions.
\begin{enumerate}
\item Two involutions $\tilde{\theta}_0$ and $\tilde{\theta}_1'$ are commutative with each other,
where $\tilde{\theta}_1'=\tau_g\circ \tilde{\theta}_1\circ \tau_g^{-1}$.
\item $\mathfrak{a}$ is also a maximal abelian subspace of $\mathfrak{p}_0\cap \mathfrak{p}_1'$,
where $\mathfrak{p}_1' = F(-\theta_1', \mathfrak{g})$.
\item Denote by $(\tilde{\Sigma},\Sigma,W)$ and $(\tilde{\Sigma}',\Sigma',W')$ the triads defined by 
$(G,K_0,K_1)$ and $(G,K_0,K_1')$, respectively, where $K_1'=F(\tilde{\theta}_1', G)_0$.
Then the following (3-1), (3-2) and (3-3) hold.
  \begin{enumerate}
  \item For any $H\in\mathfrak{a}$, $H$ is a regular point with respect to $(\tilde{\Sigma},\Sigma,W)$ if 
and only if $H$ is a regular point with respect to $(\tilde{\Sigma}',\Sigma',W')$. 
  \item For any $\lambda\in\Sigma'$, the orbit $W(\Sigma')\lambda$ spans $\mathfrak a$.
  \item For any $\alpha \in W'$, the orbit 
$W(\Sigma')\alpha$ spans $\mathfrak a$.
  \end{enumerate}
\end{enumerate}
\end{lemma}

Note that $\tilde{\Sigma}' = \tilde{\Sigma}$.
The proofs of Lemmas~\ref{lem:regularity} and \ref{lem:span} are divided into some steps. 
Since $H$ is in $\mathfrak a \subset {\mathfrak p}_0\cap {\mathfrak p}_1$, we have 
$$
\theta_0\theta_1'=\theta_0\mathrm{Ad}(a)\theta_1\mathrm{Ad}(a^{-1})=\theta_0\theta_1\mathrm{Ad}(a^{-2}).
$$
Since $\theta_0$ and $\theta_1$ commute with each other, we have 
$$
(\theta_0\theta_1')^2=\theta_0\theta_1\mathrm{Ad}(a^{-2})\theta_0\theta_1\mathrm{Ad}(a^{-2})
=\mathrm{Ad}(a^{-4}).
$$
Hence $(\theta_0\theta_1')^{2n}$ is identity on $M$ if and only if $\mathrm{Ad}(a^{4n})$ is identity on $M$. 
Here we consider the natural surjective Lie homomorphism of $\mathrm{Ad}(G)$ onto 
the subgroup $\mathrm{Ad}(G)|_{M}$ of the isometry group of $M$.
The kernel of the corresponding homomorphism between the Lie algebras is an ideal of $\mathfrak g$.  
Let $\mathfrak g=\sum {\mathfrak g}_i$ be the direct sum decomposition of $\mathfrak g$ into 
simple ideals,
and $x_0=\sum x_i$ be the corresponding decomposition of $x_0$. 
Without loss of generality we may assume that $x_i\not=0$ for each $i$. 
Then the kernel of $\mathrm{Ad}(G)\rightarrow\mathrm{Ad}(G)|_{M}$ is discrete, hence it is a covering mapping. 
Since $\mathrm{Ad}(G)$ is centerless,
the mapping $\mathrm{Ad}(G)\rightarrow \mathrm{Ad}(G)|_{M}$ 
is an isomorphism. 
In particular, $\mathrm{Ad}(g)$ is identity on $M$ if and only if $\mathrm{Ad}(g)$ is identity on $\mathfrak g$. 
Thus $(\theta_0\theta_1')^{2n}$ is identity on $M$ if and only if $\mathrm{Ad}(a^{4n})$ is identity on $\mathfrak g$.
Since $\mathrm{Ad}(a)$ is identity on 
$$
\mathfrak k (0)
\oplus \mathfrak a
\oplus V({\mathfrak k}_0\cap {\mathfrak p}_1)
\oplus V({\mathfrak k}_1\cap {\mathfrak p}_0)
= \{X\in\mathfrak g\mid [\mathfrak a,X]=\{0\}\},
$$
the condition that $\mathrm{Ad}(a^{4n})=\mathrm{id}$ on $\mathfrak{g}$ is equivalent to the condition that $\mathrm{Ad}(a^{4n})=\mathrm{id}$ on 
$$
\sum_{\lambda\in\Sigma^+}({\mathfrak k}_\lambda\oplus{\mathfrak p}_\lambda)
\oplus \sum_{\alpha\in W^+} \big( V_\alpha^\perp({\mathfrak k}_0\cap {\mathfrak p}_1)
\oplus V_\alpha^\perp({\mathfrak k}_1\cap {\mathfrak p}_0) \big),
$$
which is the orthogonal complementary subspace of $\{X\in\mathfrak g\mid [\mathfrak a,X]=\{0\}\}$. 

Take $H\in\mathfrak a$ with $a = \exp H$. 
Then $\mathrm{Ad}(a^{4n})$ is identity on $\mathfrak g$ if and only if 
$nH$ is in $\Gamma$ (see (\ref{eqn:lattice})).
Hence,
in order to prove Lemma~\ref{lem:regularity} it is sufficient to show that 
there exist $H\in {\mathfrak a}_r$ and $n\in\mathbb{N}$ such that 
$nH\in \Gamma$.
If $(\mathfrak g,\theta_0,\theta_1)$ is not irreducible
then it can be decomposed to the direct sum of irreducible factors. 

\begin{lemma}[Proposition~2.2 of \cite{Matsuki2002}] \label{lem:Matsuki}
Assume that $(\mathfrak g,\theta_0,\theta_1)$ is irreducible
and $\theta_0\theta_1=\theta_1\theta_0$.
Then $(\mathfrak g,\theta_0,\theta_1)$ is one of the following forms:
\begin{enumerate}
\item $\mathfrak g$ is simple. 
\item There exist a compact simple Lie algebra $\mathfrak u$ and an involution $\theta$ on $\mathfrak u$ such that 
$\mathfrak g=\mathfrak u\oplus\mathfrak u$ and that
$$
\theta_0(X_1,X_2)=(X_2,X_1),\quad \theta_1(X_1,X_2)=(\theta(X_2),\theta(X_1)).
$$
\item There exists a compact simple Lie algebra $\mathfrak u$ such that 
$\mathfrak g=\mathfrak u\oplus\mathfrak u\oplus \mathfrak u\oplus\mathfrak u$ and that 
\begin{align*}
& \theta_0(X_1,X_2,X_3,X_4)=(X_2,X_1,X_4,X_3),\\
& \theta_1(X_1,X_2,X_3,X_4)=(X_4,X_3,X_2,X_1).
\end{align*}
\item There exist a compact simple Lie algebra $\mathfrak u$ and an involution $\theta$ on $\mathfrak u$ such that 
$\mathfrak g=\mathfrak u\oplus\mathfrak u$ and that
$$
\theta_0(X_1,X_2)=(X_2,X_1),\quad \theta_1(X_1,X_2)=(\theta(X_1),\theta(X_2)).
$$
\item There exist a compact simple Lie algebra $\mathfrak u$ and an involution $\theta$ on $\mathfrak u$ such that 
$\mathfrak g=\mathfrak u\oplus\mathfrak u$ and that
$$
\theta_0(X_1,X_2)=(\theta(X_1),\theta(X_2)),\quad \theta_1(X_1,X_2)=(X_2,X_1).
$$
\end{enumerate}
\end{lemma}

We will show Lemmas~\ref{lem:regularity} and \ref{lem:span} in each case of Lemma~\ref{lem:Matsuki}. 
The case (5) is reduced to the case (4) if we exchange the roles of $L_0$ and $L_1$. 

In the following of this subsection,
we assume that $L_0$ and $L_1$ are not congruent.
(In a similar manner we can prove Lemmas~\ref{lem:regularity} and \ref{lem:span} in the case where $L_0$ and $L_1$ are congruent. )

\subsubsection{Case (1) of Lemma~\ref{lem:Matsuki}}

In this case,
$(G,K_0,K_1)$ determines a symmetric triad $(\tilde{\Sigma},\Sigma,W)$ of $\mathfrak a$. 
Note that
$H\in\mathfrak a$ is a regular point of $(\tilde{\Sigma},\Sigma,W)$
if and only if the intersection of $L_0$ and $L_1'$ is discrete (Theorem~\ref{thm:main}).
In this case Lemma~\ref{lem:regularity} holds by Lemma~\ref{lem:ST}.

We show Lemma~\ref{lem:span}. 
In general, if we define a new involution $\tilde{\theta}_1'$ by 
$\tilde{\theta}_1':=\tau_{\exp Y}\circ \tilde{\theta}_1\circ \tau_{\exp Y}^{-1}$ 
for $Y\in\Gamma$, then $\tilde{\theta}_0$ and $\tilde{\theta}_1'$ commute with 
each other since $\tilde{\theta}_0\tilde{\theta}_1=\tilde{\theta}_1\tilde{\theta}_0$. 
Let $(\tilde{\Sigma},\Sigma,W)$ and $(\tilde{\Sigma}',\Sigma',W')$ denote the 
symmetric triads induced from $(G,K_0,K_1)$ and $(G,K_0,K_1')$, 
respectively. 
In \cite{Ikawa2011}, Ikawa defined an equivalence relation on the set of 
symmetric triads (see Definition~2.6 in \cite{Ikawa2011}), and proved that 
$(\tilde{\Sigma},\Sigma,W)$ and $(\tilde{\Sigma}',\Sigma',W')$ are 
equivalent (see Theorem~4.33, (2) in \cite{Ikawa2011}). 
By the classification of symmetric triads (Theorem~2.19 in \cite{Ikawa2011}),
there exists $Y\in\Gamma$ such that 
$\Sigma'$ is irreducible.
Thus (1) of Lemma~\ref{lem:span} holds. 
Since $\mathfrak{p}_1'=\mathrm{Ad}(\exp Y){\mathfrak p}_1$, (2) of  Lemma~\ref{lem:span} holds. 
(3-1) of Lemma~\ref{lem:span} follows from 
${\mathfrak p}_0\cap \mathrm{Ad}(\exp H){\mathfrak p}_1'
=\mathrm{Ad}(\exp Y)({\mathfrak p}_0\cap \mathrm{Ad}(\exp H){\mathfrak p}_1)$ 
for each $H\in\mathfrak{a}$. 
(3-2) and (3-3) of Lemma~\ref{lem:span} follow from Lemma~\ref{lem:triad-span}.

\subsubsection{Case (2) of Lemma~\ref{lem:Matsuki}}

There exist a compact simple Lie algebra $\mathfrak u$ and an involution $\theta$ on $\mathfrak u$ such that 
$\mathfrak g=\mathfrak u\oplus\mathfrak u$ and that
$$
\theta_0(X_1,X_2)=(X_2,X_1),\quad \theta_1(X_1,X_2)=(\theta(X_2),\theta(X_1)).
$$
Denote by $\mathfrak u=\mathfrak k\oplus\mathfrak p$ the canonical decomposition of $\mathfrak u$ 
with respect to $\theta$. 
Then 
\begin{align*}
{\mathfrak k}_0&=\{(X,X)\mid X\in\mathfrak u\},\quad {\mathfrak p}_0=\{(X,-X)\mid X\in\mathfrak u\},\\
{\mathfrak k}_1&=\{(X,\theta(X))\mid X\in\mathfrak u\},\quad {\mathfrak p}_1=\{(X,-\theta(X))\mid X\in\mathfrak u\},\\
{\mathfrak p}_0\cap{\mathfrak p}_1&=\{(X,-X)\mid X\in\mathfrak k\}.
\end{align*}
Take a maximal abelian subalgebra $\mathfrak t$ of $\mathfrak k$. 
Then $\mathfrak a=\{(H,-H)\mid H\in \mathfrak t\}$ is a maximal abelian subspace of 
${\mathfrak p}_0\cap{\mathfrak p}_1$. 
Since $L_0$ and $L_1$ are not congruent, $\theta$ is of outer type.
Then it is known by \cite{Ikawa} that 
$(\tilde{\Sigma},\Sigma,W)$ is a symmetric triad of $\mathfrak t$.
Thus Lemma~\ref{lem:regularity} follows from Lemma~\ref{lem:ST}.
In a similar manner to Case~(1), 
Lemma~\ref{lem:span} follows from Lemma~\ref{lem:triad-span}.

\subsubsection{Case (3) of Lemma~\ref{lem:Matsuki}}

There exists a compact simple Lie algebra $\mathfrak u$ such that 
$\mathfrak g=\mathfrak u\oplus\mathfrak u\oplus \mathfrak u\oplus\mathfrak u$ and that 
$$
\theta_0(X_1,X_2,X_3,X_4)=(X_2,X_1,X_4,X_3),\quad \theta_1(X_1,X_2,X_3,X_4)=(X_4,X_3,X_2,X_1).
$$
In this case we have 
\begin{align*}
{\mathfrak k}_0&=\{(X,X,Y,Y)\mid X,Y\in\mathfrak u\},\\
{\mathfrak p}_0&=\{(X,-X,Y,-Y)\mid X,Y\in\mathfrak u\},\\
{\mathfrak k}_1&=\{(X,Y,Y,X)\mid X,Y\in\mathfrak u\},\\
{\mathfrak p}_1&=\{(X,Y,-Y,-X)\mid X,Y\in\mathfrak u\},\\
{\mathfrak p}_0\cap{\mathfrak p}_1&=\{(X,-X,X,-X)\mid X\in\mathfrak u\}.
\end{align*}
Take a maximal abelian subalgebra $\mathfrak t$ of $\mathfrak u$. 
Then 
$$
\mathfrak a=\{(H,-H,H,-H)\mid H\in\mathfrak t\}\subset {\mathfrak p}_0\cap{\mathfrak p}_1
$$
is a maximal abelian subspace of ${\mathfrak p}_0\cap{\mathfrak p}_1$. 
Let $R$ denote the root system of $\mathfrak u$ with respect to 
$\mathfrak t$. 
Then 
$$
\tilde{\Sigma} = \Sigma = W = \{(\alpha,-\alpha,\alpha,-\alpha)\mid \alpha\in R\}.
$$
Thus, in this case, Lemmas~\ref{lem:regularity} and \ref{lem:span} follow
from the theory of compact Lie groups.

\subsubsection{Case (4) of Lemma~\ref{lem:Matsuki}}

There exist a compact simple Lie algebra $\mathfrak u$ and an involution $\theta$ on $\mathfrak u$ such that 
$\mathfrak g=\mathfrak u\oplus\mathfrak u$ and that 
$$
\theta_0(X_1,X_2)=(X_2,X_1),\quad \theta_1(X_1,X_2)=(\theta(X_1),\theta(X_2)).
$$
Denote by $\mathfrak u=\mathfrak k\oplus\mathfrak p$ the canonical decomposition of 
$\mathfrak u$ with respect to $\theta$. 
Then 
\begin{align*}
{\mathfrak k}_0&=\{(X,X)\mid X\in\mathfrak u\},\quad {\mathfrak p}_0=\{(X,-X)\mid X\in\mathfrak u\},\\
{\mathfrak k}_1&=\{(X,Y)\mid X,Y\in\mathfrak k\},\quad {\mathfrak p}_1=\{(X,Y)\mid X,Y\in\mathfrak p\},\\
{\mathfrak p}_0\cap{\mathfrak p}_1&=\{(X,-X)\mid X\in\mathfrak p\}.
\end{align*}
Take a maximal abelian subspace $\mathfrak{a}'$ of $\mathfrak p$.
Then $\mathfrak{a}=\{(H,-H)\mid H\in \mathfrak{a}'\}$ is a maximal abelian subspace of 
$\mathfrak{p}_0\cap{\mathfrak p}_1$. 
Let $R$ denote the restricted root system of $\mathfrak u$ with respect to 
$\mathfrak{a}'$.
Then 
$$
\tilde{\Sigma} = \Sigma = W = \{(\alpha,-\alpha)\mid \alpha\in R\}.
$$
Thus, in this case, Lemmas~\ref{lem:regularity} and \ref{lem:span} follow
from the theory of compact symmetric spaces.

This completes the proof of Lemma~\ref{lem:span}.

\subsubsection{General case}
\label{sec:general case}

In the case where $(\mathfrak g,\theta_0,\theta_1)$ is not irreducible,
we can show Lemma~\ref{lem:regularity}
using irreducible decomposition of  $(\mathfrak g,\theta_0,\theta_1)$ 
and applying the result in the case where it is irreducible. 
We state more precisely. 

If $(\mathfrak g,\theta_0,\theta_1)$ is not irreducible,
then we decompose it to the direct sum of irreducible factors as follows: 
$$
(\mathfrak g,\theta_0,\theta_1)=\sum_{j=1}^m ({\mathfrak g}^{(j)},\theta_0^{(j)},\theta_1^{(j)}).
$$
We have two canonical decompositions of ${\mathfrak g}^{(j)}$:
$$
{\mathfrak g}^{(j)}={\mathfrak k}_0^{(j)}\oplus {\mathfrak p}_0^{(j)}={\mathfrak k}_1^{(j)}\oplus {\mathfrak p}_1^{(j)}.
$$
For $i=0,1$ we have 
$$
{\mathfrak k}_i=\sum_{j=1}^m {\mathfrak k}_i^{(j)},\quad {\mathfrak p}_i=\sum_{j=1}^m {\mathfrak p}_i^{(j)}.
$$
Take a maximal abelian subspace ${\mathfrak a}^{(j)}$ of ${\mathfrak p}_0^{(j)}\cap {\mathfrak p}_1^{(j)}$. 
Then $\displaystyle{\mathfrak a=\sum_{j=1}^m {\mathfrak a}^{(j)}}$ is a maximal abelian subspace of 
${\mathfrak p}_0\cap {\mathfrak p}_1$. 
Each $({\mathfrak g}^{(j)},\theta_0^{(j)},\theta_1^{(j)})$ defines 
a triad $(\tilde{\Sigma}^{(j)},\Sigma^{(j)},W^{(j)})$, and each 
$(\tilde{\Sigma}^{(j)},\Sigma^{(j)},W^{(j)})$ defines the set ${\mathfrak a}^{(i)}_r$of 
regular points of ${\mathfrak a}^{(i)}$. 
Set $\Gamma^{(j)}=\{X\in {\mathfrak a}^{(j)}\mid \langle\lambda ,X\rangle 
\in \frac{\pi}{2}\mathbb{Z} \mbox{ for all }\lambda\in\tilde{\Sigma}^{(j)}\}$. 
Then 
\begin{align*}
& \tilde{\Sigma}=\bigcup_{j=1}^m \tilde{\Sigma}^{(j)},\quad 
\Sigma = \bigcup_{j=1}^m \Sigma^{(j)}, \quad 
W = \bigcup_{j=1}^m W^{(j)},\quad 
 {\mathfrak a}_r=\sum_{j=1}^m {\mathfrak a}_r^{(i)} ,\quad 
\Gamma = \sum_{j=1}^m \Gamma^{(j)}.
\end{align*}
As we mentioned above there exist $H^{(j)}\in {\mathfrak a}_r^{(j)}$ and 
$n_j\in\mathbb{N}$ such that $n_jH^{(j)}\in \Gamma^{(j)}$. 
If we set $H=\sum H^{(j)}\in {\mathfrak a}_r$ and $n=n_1\cdots n_m\in\mathbb{N}$, 
then $nH$ is in $\Gamma$. 
Thus Lemma~\ref{lem:regularity} holds.

\section{Examples}
\label{sec:examples}

By Theorems~\ref{thm:complexflag} and \ref{thm:main}, we can explicitly describe
a maximal antipodal set of a complex flag manifold
and the intersection of two real flag manifolds.
In this section, we demonstrate some concrete examples,
where two real flag manifolds $L_0$ and $L_1$ are not congruent.
See \cite{IST2014} and \cite{IST-ICM2014} for examples where $L_0$ and $L_1$ are congruent.

\subsection{The case $(G,K_0,K_1)=(SU(2n), SO(2n), Sp(n))$}

Let $G = SU(2n)$.
Define an $\mathrm{Ad}(G)$-invariant inner product on $\mathfrak{g} = \mathfrak{su}(2n)$ normalizing as
\begin{equation} \label{eq:invariant inner product on su(m)}
\langle X,Y \rangle = -\frac{1}{8n}B(X,Y) = -\frac{1}{2}\mathrm{tr }(XY) \qquad (X,Y \in \mathfrak{g}),
\end{equation}
where $B$ is the Killing form of $\mathfrak{g}$.
We give involutions $\tilde\theta_0$ and $\tilde\theta_1$ on $G$ by
$$
\tilde\theta_0(g) = \bar{g},\qquad
\tilde\theta_1(g) = J_n^{-1} \bar{g} J_n \qquad (g \in G),
$$
where $J_n := \left[\begin{array}{cc} O & -I_n \\ I_n & O \end{array} \right]$
and $I_n$ is the identity matrix of $n \times n$.
Since $\tilde\theta_0 \tilde\theta_1 = \tilde\theta_1 \tilde\theta_0$,
the Lie algebra $\mathfrak{g} = \mathfrak{su}(2n)$ is decomposed as (\ref{eqn:canonical decompositions})
and (\ref{eq:simultaneous decomposition}), where
\begin{align*}
\mathfrak{k}_0 &= \mathfrak{so}(2n), \\
\mathfrak{p}_0 &= \{ \sqrt{-1}Z \mid Z \in M_{2n}(\mathbb{R}),\ {}^tZ=Z,\ \mathrm{tr }Z=0 \}, \\
\mathfrak{k}_1 &= \mathfrak{sp}(n), \\
\mathfrak{p}_1 &= \left\{ \left[\begin{array}{cc} X & Y \\ \bar{Y} & -\bar{X} \end{array} \right] \ \bigg| \
\begin{array}{l} X,Y \in M_n(\mathbb{C}), \\ {}^t\bar{X}=-X,\ {}^tY=-Y,\ \mathrm{tr }X=0 \end{array} \right\}.
\end{align*}
Take a maximal abelian subspace $\mathfrak{a}$ in 
$$
\mathfrak{p}_0 \cap \mathfrak{p}_1
= \left\{ \sqrt{-1} \left[ \begin{array}{cc} X & Y \\ -Y & X \end{array} \right] \ \bigg| \
\begin{array}{l} X,Y \in M_n(\mathbb{R}), \\ {}^tX = X,\ {}^tY=-Y,\ \mathrm{tr }X=0 \end{array} \right\}
$$
as
$$
\mathfrak{a}
= \left\{ H = \sqrt{-1} \left[ \begin{array}{cc} X & O \\ O & X \end{array} \right] \ \bigg| \
\begin{array}{l} X = \mathrm{diag}(t_1, \ldots, t_n), \\ t_1, \ldots, t_n \in \mathbb{R},\
t_1+\cdots+t_n=0 \end{array} \right\}.
$$
In this case, by a direct calculation, we verify
$$
\tilde{\Sigma} = \Sigma = W = \{ \pm(e_i-e_j) \mid 1 \leq i < j \leq n \},
$$
where $e_i-e_j \in \mathfrak{a} \ (i \neq j)$ is defined by
$\langle e_i-e_j, H \rangle = t_i-t_j$ for all $H \in \mathfrak{a}$.
Hence the triad $(\tilde{\Sigma}, \Sigma, W)$ is of type $\text{III-}A_{n-1}$
in the classification of symmetric triads (see \cite[p.~92]{Ikawa2011}).
The set of regular points is
$$
\mathfrak{a}_r = \left\{ H \in \mathfrak{a} \ \Big| \ \langle e_i-e_j, H \rangle \not\in \frac{\pi}{2}\mathbb{Z}
\ (1 \leq i < j \leq n) \right\}.
$$

Now we take a nonzero vector
$$
x_0 = \sqrt{-1} \left[ \begin{array}{cc} X & O \\ O & X \end{array} \right] \in \mathfrak{a}
$$
where $X = \mathrm{diag}(x_1I_{n_1}, \ldots, x_{r+1}I_{n_{r+1}})$ and
$x_i$ are distinct real numbers satisfying $n_1x_1+\cdots+n_{r+1}x_{r+1} = 0$.
Then real flag manifolds
\begin{align*}
L_0 &= \mathrm{Ad}(K_0)x_0
\cong SO(2n)/S(O(2n_1) \times \cdots \times O(2n_{r+1}))
\cong F_{2n_1, \ldots, 2n_r}(\mathbb{R}^{2n}), \\
L_1 &= \mathrm{Ad}(K_1)x_0
\cong Sp(n)/Sp(n_1) \times \cdots \times Sp(n_{r+1})
\cong F_{n_1, \ldots, n_r}(\mathbb{H}^{n})
\end{align*}
are embedded as real forms in the complex flag manifold
$$
M = \mathrm{Ad}(G)x_0
\cong SU(2n)/S(U(2n_1) \times \cdots \times U(2n_{r+1}))
\cong F_{2n_1, \ldots, 2n_r}(\mathbb{C}^{2n}).
$$
Here, for positive integers $n, n_1, \ldots, n_r$ which satisfy $n_{r+1} := n - (n_1 + \cdots + n_r) > 0$,
let $F_{n_1, \ldots, n_r}(\mathbb{K}^n)$ denote the flag manifold of sequences of $\mathbb{K}$-subspaces in $\mathbb{K}^n$,
where $\mathbb{K} = \mathbb{R}, \mathbb{C}$ or $\mathbb{H}$, that is,
$$
F_{n_1, \ldots, n_r}(\mathbb{K}^n)
= \left\{ (V_1, \ldots, V_r) \ \Big| \
\begin{array}{l}
V_1 \subset \cdots \subset V_r \subset \mathbb{K}^n : \text{$\mathbb{K}$-subspaces} \\
\dim_{\mathbb{K}} V_l = n_1 + \cdots + n_l \ (l=1, \ldots, r)
\end{array} \right\},
$$
which was originally called a flag manifold.

By Theorem~\ref{thm:main},
for $a = \exp H \ (H \in \mathfrak{a})$,
the intersection $L_0 \cap \mathrm{Ad}(a)L_1$ is discrete
if and only if $H \in \mathfrak{a}_r$,
and then
$$
L_0 \cap \mathrm{Ad}(a)L_1
= M \cap \mathfrak{a}
= W(\tilde\Sigma)x_0
= W(R_0)x_0 \cap \mathfrak{a}
= W(R_1)x_0 \cap \mathfrak{a}.
$$
Note that, in this case, $\mathfrak{a}$ is also a maximal abelian subspace in $\mathfrak{p}_1$,
and $\tilde\Sigma = R_1$.
The root systems $\tilde\Sigma$ and $R_1$ are of type $A_{n-1}$,
and their Weyl groups $W(\tilde\Sigma)$ and $W(R_1)$
act on $\mathfrak{a}$ as permutations of diagonal elements of $X = \mathrm{diag}(t_1, \ldots, t_n)$.
The intersection $L_0 \cap \mathrm{Ad}(a)L_1$ can be described
in the flag model $F_{2n_1, \ldots, 2n_r}(\mathbb{C}^{2n})$ as follows.
In $\mathbb{C}^{2n}$ we define $i,j,k$ by
$$
iv = \sqrt{-1}v, \quad
jv = J_n \bar{v}, \quad
kv = ijv \quad (v \in \mathbb{C}^{2n}).
$$
Then $\mathbb{C}^{2n}$ can be identified with $\mathbb{H}^n$.
This identification gives an embedding of $F_{n_1, \ldots, n_r}(\mathbb{H}^n)$
into $F_{2n_1, \ldots, 2n_r}(\mathbb{C}^{2n})$.
Let $e_1, \ldots, e_{2n}$ be the standard basis of $\mathbb{C}^{2n}$,
and let $W_i := \langle e_i, e_{n+i} \rangle_\mathbb{C} = \langle e_i \rangle_\mathbb{H} \ (1 \leq i \leq n)$.
For $a = \exp H \ (H \in \mathfrak{a}_r)$,
the intersection of $F_{2n_1, \dots, 2n_r}(\mathbb R^{2n})$
and $aF_{n_1, \dots, n_r}(\mathbb H^n)$ is
\begin{align*}
&
F_{2n_1, \dots, 2n_r}(\mathbb R^{2n}) \cap aF_{n_1, \dots, n_r}(\mathbb H^n) \\
& = \{(
W_{i_1} \oplus \cdots \oplus W_{i_{n_1}},
W_{i_1} \oplus \cdots \oplus W_{i_{n_1+n_2}}, \dots,
W_{i_1} \oplus \cdots \oplus W_{i_{n_1+\cdots+n_r}}) \\
& \quad \mid 1 \le i_1 < \cdots < i_{n_1} \le n,\,
1 \le i_{n_1+1} < \cdots < i_{n_1+n_2} \le n, \dots, \\
& \qquad 1 \le i_{n_1+\cdots+n_{r-1}+1} < \cdots < i_{n_1+\cdots+n_r} \le n, \\
& \qquad i_1, i_2, \dots, i_{n_1+\cdots+n_r} \text{ are mutually distinct}
\},
\end{align*}
which is an antipodal set of $F_{2n_1, \dots, 2n_r}(\mathbb C^{2n})$.
Indeed, by Theorem~\ref{thm:complexflag}, 
\begin{align*}
& \{(\langle e_{i_1}, \dots, e_{i_{2n_1}}\rangle_{\mathbb C},
\langle e_{i_1}, \dots, e_{i_{2n_1+2n_2}}\rangle_{\mathbb C}, \dots, 
\langle e_{i_1}, \dots, e_{i_{2n_1+\cdots+2n_r}}\rangle_{\mathbb C}) \\
& \quad \mid 1 \leq i_1 < \cdots < i_{2n_1} \leq 2n,\,
1 \leq i_{2n_1+1} < \cdots < i_{2n_1+2n_2} \leq 2n, \dots, \\
& \qquad 1 \leq i_{2n_1+\cdots+2n_{r-1}+1} < \cdots < i_{2n_1+\cdots+2n_r} \leq 2n, \\
& \qquad i_1, i_2, \dots, i_{2n_1+\cdots+2n_r} \text{ are mutually distinct}
\}
\end{align*}
is a maximal antipodal set of $F_{2n_1, \dots, 2n_r}(\mathbb C^{2n})$.
In particular, we have
$$
\#(L_0 \cap L_1')
= \frac{n!}{n_1!n_2!\cdots n_{r+1}!},
$$
where $L_1' = \mathrm{Ad}(a)L_1$ for $a = \exp H \ (H \in \mathfrak{a}_r)$.
Note that the cardinality $\#(L_0 \cap L_1')$ is equal to $SB(L_1;\mathbb{Z}_2)$ in this case (see p.~202 of \cite{Borel53}).

\subsection{The case $(G,K_0,K_1)=(SU(n), SO(n), S(U(1)\times U(n-1)))$}

Let $G =SU(n)$.
Define an $\mathrm{Ad}(G)$-invariant inner product $\langle \ , \ \rangle$ on $\mathfrak{g} = \mathfrak{su}(n)$
as (\ref{eq:invariant inner product on su(m)}).
We give involutions $\tilde\theta_0$ and $\tilde\theta_1$ on $G = SU(n)$ by
$$
\tilde\theta_0(g) = \bar{g}, \qquad
\tilde\theta_1(g) = I_{1,n-1} g I_{1,n-1} \qquad (g \in G)
$$
where $I_{1,n-1} = \left[ \begin{array}{c|c} 1 & \\ \hline & -I_{n-1} \end{array} \right]$.
Since $\tilde\theta_0 \tilde\theta_1 = \tilde\theta_1 \tilde\theta_0$,
the Lie algebra $\mathfrak{g} = \mathfrak{su}(n)$ is decomposed as (\ref{eqn:canonical decompositions})
and (\ref{eq:simultaneous decomposition}), where
\begin{align*}
\mathfrak{k}_0 &= \mathfrak{so}(n), \\
\mathfrak{p}_0 &= \{ \sqrt{-1}Y \mid Y \in M_n(\mathbb{R}), {}^tY = Y,\ \mathrm{tr } Y=0 \}, \\
\mathfrak{k}_1 &= \mathfrak{s}\big( \mathfrak{u}(1) \oplus \mathfrak{u}(n-1) \big), \\
\mathfrak{p}_1 &= \left\{ \left[ \begin{array}{c|c} & -^t\bar{Z} \\ \hline Z & \end{array} \right] 
\ \Big| \ Z \in \mathbb{C}^{n-1} \right\}.
\end{align*}
Take a maximal abelian subspace of
$$
\mathfrak{p}_0 \cap \mathfrak{p}_1
= \left\{ \sqrt{-1} \left[ \begin{array}{c|c} & {}^tY \\ \hline Y & \end{array} \right]
\ \Big| \ Y \in \mathbb{R}^{n-1} \right\}
$$
as
$$
\mathfrak{a}
= \left\{ H = \sqrt{-1} \left[ \begin{array}{c|c} & {}^tY \\ \hline Y & \end{array} \right]
\ \Big| \ Y = {}^t(s,0, \ldots, 0), s \in \mathbb{R} \right\}.
$$
In this case, by a direct calculation, we verify
$$
\Sigma = \{ \pm\alpha \},\quad
W = \{ \pm\alpha, \pm 2\alpha \},\quad
\tilde\Sigma = \Sigma \cup W = \{ \pm\alpha, \pm 2\alpha \},
$$
where $\alpha = 
\sqrt{-1} \left[ \begin{array}{c|c} & {}^tY \\ \hline Y & \end{array} \right]$
and $Y = {}^t(1,0,\ldots,0)$.
Hence the triad $(\tilde{\Sigma}, \Sigma, W)$ is of type $\text{II-}BC_1$
in the classification of symmetric triads (see p.~91 of \cite{Ikawa2011}).
The set of regular points is
\begin{align*}
\mathfrak{a}_r
&= \bigcap_{\lambda\in\Sigma\atop\beta\in W}
\left\{ H \in \mathfrak{a} \ \Big| \
\langle\lambda, H\rangle \not\in \pi\mathbb{Z},\,
\langle\beta, H\rangle \not\in \frac{\pi}{2} + \pi\mathbb{Z} \right\} \\
&= \left\{ H \in \mathfrak{a} \ \Big| \
\langle\alpha, H\rangle \not\in \frac{\pi}{4}\mathbb{Z} \right\}.
\end{align*}

Now we take $x_0 = 
\sqrt{-1} \left[ \begin{array}{c|c} & {}^tY \\ \hline Y & \end{array} \right]
\in \mathfrak{a}$
where $Y = {}^t(1,0,\ldots,0)$.
Then real flag manifolds
\begin{align*}
L_0 &= M \cap \mathfrak{p}_0 = \mathrm{Ad}(K_0)x_0 \cong SO(n)/S(O(1) \times O(1) \times O(n-2)) \cong F_{1,1}(\mathbb{R}^n),\\
L_1 &= M \cap \mathfrak{p}_1 = \mathrm{Ad}(K_1)x_0
\cong S^{2n-3}
\end{align*}
are embedded as real forms in the complex flag manifold
$$
M = \mathrm{Ad}(G) x_0 \cong SU(n)/S(U(1) \times U(1) \times U(n-2)) \cong F_{1,1}(\mathbb{C}^n).
$$
Note that $L_1$ is diffeomorphic to $S^{2n-3}$
since it is an orbit of the linear isotropy representation of
the symmetric pair $(SU(n), S(U(1)\times U(n-1))$ of rank one.

By Theorem~\ref{thm:main},
for $a = \exp H \ (H \in \mathfrak{a})$,
the intersection $L_0 \cap \mathrm{Ad}(a)L_1$ is discrete
if and only if $H \in \mathfrak{a}_r$, that is, $s \not\in \frac{\pi}{4}\mathbb{Z}$.
In this case
$$
L_0 \cap \mathrm{Ad}(a)L_1
= M \cap \mathfrak{a}
= W(\tilde\Sigma)x_0
= \{ \pm x_0 \}.
$$
In the flag model $F_{1,1}(\mathbb{C}^n)$, the intersection is isometric to
$$
\big\{ \big( \langle e_1 \rangle_{\mathbb{C}}, \langle e_1, e_2 \rangle_{\mathbb{C}} \big),\
\big( \langle e_2 \rangle_{\mathbb{C}}, \langle e_1, e_2 \rangle_{\mathbb{C}} \big) \big\}.
$$
In particular, we have
$\#(L_0 \cap L_1') = 2 = SB(L_1; \mathbb{Z}_2)$,
where $L_1' = \mathrm{Ad}(a)L_1$ for $a = \exp H \ (H \in \mathfrak{a}_r)$.

\section*{Acknowledgment}
\label{sec:Acknowledgment}

The authors would like to thank Professor Yoshihiro Ohnita for his kind explanation of the paper \cite{Ohnita} for them.
This work was partly supported by MEXT Promotion of Distinctive Joint Research Center Program JPMXP0723833165.
The first author was supported by JSPS KAKENHI Grant Number
JP16K05128 and JP22K03285.
The second author was supported by JSPS KAKENHI Grant Number
JP16K05120 and JP20K03576.
The third author was supported by JSPS KAKENHI Grant Number
JP24K06714.
The fourth author was supported by JSPS KAKENHI Grant Number JP21K03250.
The fifth author was supported by JSPS KAKENHI Grant Number
JP21K03218.


\end{document}